\documentclass[a4paper,10pt, reqno]{amsart}

\usepackage{amsmath,amssymb,latexsym, amsfonts}
\usepackage{verbatim}
\usepackage{times}
\usepackage{mathbbol}
\usepackage{hyperref}
\usepackage{mathabx}
\usepackage{bbm}

\usepackage{tikz}

\usepackage[all,ps]{xy}
\usepackage{color}

\usepackage{stmaryrd}

\usepackage{capt-of}

\usepackage[OT2, T1]{fontenc} 

\usepackage[main=english, russian]{babel}

\numberwithin{equation}{section}

\theoremstyle{plain}

\newtheorem{theorem}{Theorem}[section]

\newtheorem{prop}[theorem]{Proposition}

\newtheorem{lem}[theorem]{Lemma}

\newtheorem{cor}[theorem]{Corollary}

\theoremstyle{definition}

\newtheorem{dfn}[theorem]{Definition}
\newtheorem{example}[theorem]{Example}

\newtheorem{rem}[theorem]{Remark}

\newcommand{\emme}{{\scriptscriptstyle{M}}}

 \newcommand{\N}{{\mathbb{N}}}

 \newcommand{\K}{{\mathbb{K}}}

\newcommand{\gd}{\delta} 
\newcommand{\gD}{\Delta} 
\newcommand{\gve}{\varepsilon} 
  
\newcommand{\gvf}{\varphi}

\newcommand{\gs}{\sigma} 
\newcommand{\gS}{\Sigma}

\newcommand{\cO}{{\mathcal O}}

\newcommand{\Hom}{\operatorname{Hom}}

\newcommand{\ad}{{\operatorname{ad}}}
\newcommand{\coad}{{\operatorname{{coad}}}}
\newcommand{\alg}{{\operatorname{{alg}}}}
\newcommand{\coalg}{{\operatorname{{coalg}}}}

\newcommand{\Tor}{{\rm Tor}}
\newcommand{\Ext}{{\rm Ext}}

\newcommand{\id}{{\rm id}}

\newcommand{\Tot}{{\rm Tot}\,}

\newcommand{\pl}{\partial}

\newcommand{\rmref}[1]{{\rm (}\ref{#1}{\rm )}}

\newcommand{{\Hl}}{{H^{\ell}}} 
\newcommand{{\mHop}}{{m_{H^{\rm op}}}} 
\newcommand{{\Hop}}{{H^{\rm op}}} 
\newcommand{{\mUop}}{{m_{U^{\rm op}}}} 
\newcommand{{\mUopp}}{{m_{\scriptscriptstyle{U^{\rm op}}}}} 
\newcommand{{\Uop}}{{U^{\rm op}}}
\newcommand{{\mVop}}{{m_{V^{\rm op}}}} 
\newcommand{{\Vop}}{{V^{\rm op}}}  
\newcommand{{\Ae}}{{A^{\rm e}}}
\newcommand{{\Be}}{{B^{\rm e}}}
\newcommand{{\Ue}}{{U^{\rm e}}}
\newcommand{{\He}}{{H^{\rm e}}}
\newcommand{{\Aop}}{{A^{\rm op}}}
\newcommand{{\Aope}}{({A^{\rm op}})^{\rm e}}
\newcommand{{\Aopl}}{{A^{\rm op}_\pl}}

\newcommand{{\Bop}}{{B^{\rm op}}}
\newcommand{{\Bopp}}{{\scriptscriptstyle{{B^{\rm op}}}}}
\newcommand{{\Bope}}{({B^{\rm op}})^{\rm e}}
\newcommand{{\Bpl}}{{B_\pl}}

\newcommand{{\op}}{{{\rm op}}}
\newcommand{{\coop}}{{{\rm coop}}}
\newcommand{{\sop}}{{*^{\rm op}}}
\newcommand{{\co}}{{{\rm co}}}
\newcommand{{\diag}}{{{\rm diag}}}

\newcommand{\kmod}{\K\mbox{-}\mathbf{Mod}}                     %
\newcommand{\hmod}{H\mbox{-}\mathbf{Mod}}                     %

\newcommand{\comodh}{\mathbf{Comod}\mbox{-}H}

\newcommand{\lact}{\smalltriangleright}                  
\newcommand{\ract}{\smalltriangleleft}

\newcommand{{\gog}}{{G \rightrightarrows G_0}}

\newcommand{{\rra}}{\rightrightarrows}

\newcommand{{\lra}}{\ \longrightarrow \ }
\newcommand{{\lla}}{\ \longleftarrow \ }
\newcommand{{\lma}}{\ \longmapsto \ }

\newcommand{{\bull}}{{\scriptscriptstyle{\bullet}}}
\newcommand{{\qqquad}}{{\quad\quad\quad}}

\newsavebox{\foobox}

\newcommand{\blue}[1]{{\color{blue}{#1}}}

\begin{document}

\title{Cyclic Gerstenhaber-Schack cohomology} 

\author{Domenico Fiorenza}
\author{Niels Kowalzig}

\begin{abstract}
We show that the diagonal complex computing the Gerstenhaber-Schack cohomology of a bialgebra (that is, the cohomology theory governing bialgebra deformations) can be given the structure of an operad with multiplication if the bialgebra is a (not necessarily finite dimensional) Hopf algebra with invertible antipode; if the antipode is involutive, the operad is even cyclic. 
Therefore, the Gerstenhaber-Schack cohomology of any such Hopf algebra carries a Gerstenhaber resp.\ Batalin-Vilkoviski\u\i\ algebra structure; in particular, one obtains a cup product and a cyclic boundary $B$ that generate the Gerstenhaber bracket, and that allows to define cyclic Gerstenhaber-Schack cohomology. 
In case the Hopf algebra in question is finite dimensional, the Gerstenhaber bracket turns out to be zero in cohomology and hence the interesting structure is not given by this $e_2$-algebra structure but rather by the resulting $e_3$-algebra structure, which is expressed in terms of the cup product and $B$. 
\end{abstract}

\address{Dipartimento di Matematica, Universit\`a degli Studi di Roma La
Sapienza, P.le Aldo Moro 5, 00185 Roma, Italia}
\email{fiorenza@mat.uniroma1.it}

\address{Dipartimento di Matematica, Universit\`a degli Studi di Napoli Federico II, Via Cintia, 80126 Napoli, Italia}
\email{niels.kowalzig@unina.it}



\maketitle

\setcounter{tocdepth}{1}
\tableofcontents

\section{Introduction}
\subsection{Aims and objectives}
\label{firstofall}
Much as Hochschild cohomology for an associative algebra characterises deformations of the product structure \cite{Ger:TCSOAAR}, Gerstenhaber-Schack cohomology characterises bialgebra deformations \cite{GerSch:ABQGAAD}, that is, it occurs when one wants to simultaneously deform both the multiplication as well as the comultiplication, maintaining the compatibility between the two of them. More precisely, for $H$ a bialgebra over a field of characteristic zero $\K$ (assumption that can actually be relaxed), one defines the Gerstenhaber-Schack bicomplex of $H$ as $C^{pq}_{GS}(H,H)=\Hom_\K(H^{\otimes p},H^{\otimes q})$ in degree $(p,q)$, the columns of which are given by the Hochschild cochain complex which uses products and actions and the rows by the coHochschild (or {\em Cartier}) cochain complex which, in turn, uses coproducts and coactions. The \emph{Gerstenhaber-Schack cohomology} of $H$, denoted $H^\bullet_{GS}(H,H)$, is the cohomology of the total complex associated to this bicomplex. As $C^{\bullet\bullet}_{GS}(H,H)$ is a bicomplex associated to a bi-cosimplicial $\K$-module via the Dold-Kan correspondence, by the Dold-Puppe generalisation of the Eilenberg-Zilber theorem the cohomology of the total complex can be computed by the diagonal complex $(C^\bullet_{\diag}(H,H),\delta^\diag)$, where $C^n_{\diag}(H,H) = \Hom_\K(H^{\otimes n}, H^{\otimes n})$. 

A classical result tells that Hochschild cohomology carries a higher structure in the form of a Gerstenhaber algebra \cite{Ger:TCSOAAR}, that is, a Lie bracket of degree $-1$ plus a graded commutative cup product that are compatible in a graded Leibniz sense. By the above deformation analogy, it is therefore natural to ask whether such a structure also exists on Gerstenhaber-Schack cohomology, a question 
which
has been raised several times in the literature before (see, for example,  \cite{FarSol:GSOTCOHA, GerSch:ABQGAAD, Tai:IHBCOIDHAAGCOTYP}, and possibly others) but apparently so far no construction has been found that gives a nontrivial bracket.

In this article, we show that when $H$ is a Hopf algebra, the collection of $\K$-modules $C^\bullet_{\diag}(H,H)$ carries a natural structure of an operad with multiplication whose associated Hochschild-type 
complex (that is, originating from its cosimplicial structure) is the diagonal complex mentioned above. By a classical result \cite{GerSch:ABQGAAD, McCSmi:ASODHCC}, this implies in particular that the cohomology of $(C^\bullet_{\diag}(H,H),\delta^\diag)$ (hence, the Gerstenhaber-Schack cohomology of $H$) carries the structure of a Gerstenhaber algebra. This bracket is in general nonzero: for instance, if $H=\mathcal{U}_{\mathfrak{g}}$, the universal enveloping algebra of a Lie algebra $\mathfrak{g}$, the Gerstenhaber bracket between two $1$-cocycles $f$ and $g$ reads
\begin{equation}
\label{dayne}
\{f,g\}(XY)=[g(X),f(Y)]-[f(X),g(Y)], \qquad \forall X,Y\in \mathfrak{g},
\end{equation}
while there are no nonzero $1$-coboundaries.
If moreover the antipode of $H$ is involutive (or, somewhat weaker, the Hopf algebra is endowed with a modular pair in involution), then $C^\bullet_{\diag}(H,H)$ is actually a cyclic operad with multiplication, and hence the Gerstenhaber algebra structure on its cohomology is part of a 
Batalin-Vilkoviski\u\i\
(BV) algebra structure as then automatically follows
by a well-known result of Menichi \cite[Thm.~1.4]{Men:BVAACCOHA}.

On the other hand, when $H$ is finite dimensional, the aforementioned Gerstenhaber bracket on $H^\bullet_{GS}(H,H)$ vanishes. Although the respective Gerstenhaber brackets are constructed in quite a different manner, this fact somehow mirrors an analogous observation made by Taillefer 
\cite[Appendix A]{Tai:IHBCOIDHAAGCOTYP}. As a consequence, by Theorem 5.7 in \cite{FioKow:HBOCANCC},  
one obtains an induced $e_3$-algebra structure on Gerstenhaber-Schack cohomology. 

The different behaviour of the Gerstenhaber bracket in the finite dimensional and infinite dimensional case may sound surprising, but it actually has a transparent explanation in terms of Hopf algebra theory: a finite dimensional Hopf algebra exhibits no nonzero primitive elements. The  $e_3$-algebra structure constructed this way on the Gerstenhaber-Schack cohomology of a finite dimensional Hopf algebra is presumably different from the 
ones obtained by higher categorical methods in \cite{GinYal:DTOBHHCAF, Sho:DGCADC, Sho:TOABFATMC} as the latter do not rely on the finite dimensionality of $H$. Nevertheless, even our approach does not seem to work if $H$ is merely a bialgebra, that is, without an antipode, the necessity of which was conjectured by Shoikhet in his approach \cite[p.~9]{Sho:TOABFATMC}.

In the general bialgebra case, one gets an affirmative answer to the existence of Gerstenhaber algebra structures on Gerstenhaber-Schack cohomologies by considering a variant of the latter; namely, by replacing $C^{\bullet\bullet}_{GS}(H,H)$ with the $(1,1)$-shifted double complex $\overline{C}^{\bullet\bullet}_{GS}(H,H)$
having $\Hom_\K(H^{\otimes p},H^{\otimes q})$ in bidegree $(p-1,q-1)$, for $p,q\geq 1$. 
The question whether this (bi-)shifted bicomplex 
has a Gerstenhaber-like bracket $[-,-]$ such that $m + \Delta$ satisfies the Maurer-Cartan equation $[m + \Delta, m + \Delta] = 0$ if and only if $(m, \Delta)$ is a bialgebra structure on $H$ and such that the differential of the total complex of $\overline{C}^{\bullet\bullet}_{GS}(H,H)$ is given by $[m + \Delta,-]$, has been solved by Markl \cite{Mar:Intrinsic} by setting the problem in the more general and flexible context of $\mathrm{Lie}_\infty$-algebras.  More precisely, $\overline{C}^{\bullet\bullet}_{GS}(H,H)$ is endowed in {\em op.~cit.} with a $\mathrm{Lie}_\infty$-algebra structure such that the Lie bracket $[-,-]$ almost governs the bialgebra structure: the quadratic equation $[m + \Delta, m + \Delta] = 0$ recovers both the associativity of the multiplication $m$ and the coassociativity of the comultiplication $\Delta$, but not the compatibility between them, 
which, however, is read off by the ``quantum Maurer-Cartan equation'' involving the higher Lie brackets in the $\mathrm{Lie}_\infty$-algebra structure on $\overline{C}^{\bullet\bullet}_{GS}(H,H)$.

\subsection{Main results}
We refer to the main text for notation, terminology, and all details used in this introductory presentation.

In \S\ref{steckdose}, we obtain a cocyclic structure on the diagonal complex by considering 
 in the Gerstenhaber-Schack bicomplex
first the columns, then the rows, and finally restricting to the diagonal. The somewhat surprising observation is here that both columns and rows, that is, the Hochschild resp.\ coHochschild (Cartier) cochain complexes are in general only para-cocyclic
(that is, the cyclic operator does not power to the identity after $n+1$ steps), whereas this restriction disappears on the diagonal. More precisely, the Hochschild (and likewise the coHochschild) cochain complex for a $\K$-algebra usually does {\em not} admit a cyclic structure (see \cite[Thm.~1.3]{Kow:WEIABVA} for a sufficiency criterion) but the presence of additional structure maps in the Hopf case does allow at least for a para-cocyclic operator $\tau_v$ (resp.\ $\tau_h$), which only in case of cocommutativity of the underlying coalgebra becomes truly cocyclic. 
This way, one obtains the structure of a bi-para-cocyclic $\K$-module on the Gerstenhaber-Schack bicomplex. What is more, a direct computation yields the additional property
\begin{equation}
\label{khatiabuniatishvili}
\tau_h^{n+1} \circ \tau_v^{n+1} =  \tau_v^{n+1} \circ \tau_h^{n+1} = \id
\end{equation}
in degree $n$, which does {\em not} require cocommutativity any more but only 
the antipode to be involutive (or rather equipped with a modular pair in involution).
A bi-para-cocyclic $\K$-module that fulfils \rmref{khatiabuniatishvili} is called {\em cylindrical} 
in \cite[p.~164]{GetJon:TCHOCPA}. Defining then
$$
\tau_\diag  := \tau_\alg \circ \tau_\coalg = \tau_\coalg \circ \tau_\alg
$$
obviously yields a truly cocyclic operator on the diagonal cochain complex and enables us to prove in Theorem \ref{backspace}:

{\renewcommand{\thetheorem}{{A}}
\begin{theorem}
\label{A}
For any Hopf algebra over a field with involutive antipode, the Gerstenhaber-Schack double complex defines a cylindrical $\K$-module and its associated diagonal complex a cocyclic $\K$-module.
\end{theorem}
}

This allows to define  {\em cyclic Gerstenhaber-Schack cohomology} as in Remark \ref{cyclicgs}: the diagonal of the Gerstenhaber-Schack bicomplex can be made into a mixed complex the standard way, the total cohomology of which gives cyclic cohomology, and in particular an operator $B: H^{\bullet}_{GS}(H,H) \to H^{\bullet-1}_{GS}(H,H)$ that squares to zero.

In \S\ref{waldkonzert}, we shall approach the Gerstenhaber-Schack diagonal complex from 
the operadic point of view, which as mentioned in \S\ref{firstofall} by a classical construction directly leads to higher structures in cohomology. 
More precisely, in Theorem \ref{japonvuduciel1} and Corollary \ref{fegato}, we show:

{\renewcommand{\thetheorem}{{B}}
\begin{theorem}
\label{B}
%
%
%
For a not necessarily finite dimensional Hopf algebra with invertible antipode, 
the diagonal Gerstenhaber-Schack complex carries the structure of an operad with multiplication, which induces the structure of a Gerstenhaber algebra on Gerstenhaber-Schack (or deformation) cohomology. 
\end{theorem}
}

The (unfortunately too naive) idea that the operadic structure deployed in this theorem can be simply obtained by intertwining the classical ones for the endomorphism resp.\ coendomorphism operad of the first column resp.\ row in \rmref{tauschetasche}, possibly generalised to more general coefficients in the higher columns resp.\ rows, does not work: the precise composition law requires the existence of an antipode and hence, whereas the Gerstenhaber-Schack bicomplex can be perfectly defined if the underlying object is merely a $\K$-bialgebra, this does not seem to be the case for the operadic structure on the diagonal complex.

As a pleasant surprise, the cocyclic operator $\tau_\diag$ turns out to be compatible with the operadic structure in Theorem \ref{B} as well as with the multiplication element, which leads to the following result 
in Theorem \ref{japonvuduciel2} and Corollary \ref{nikolaus}:

{\renewcommand{\thetheorem}{{C}}
\begin{theorem}
\label{C}
For a not necessarily finite dimensional Hopf algebra with involutive antipode (or with a modular pair in involution), 
the diagonal Gerstenhaber-Schack complex carries the structure of a cyclic operad with multiplication which induces the structure of a BV algebra on Gerstenhaber-Schack cohomology. 
\end{theorem}
}

As already mentioned above, in case the Hopf algebra is finitely generated as a module over $\K$, the situation changes substantially as then the Gerstenhaber bracket vanishes in cohomology. In this situation, the interesting structure is of {\em even higher} nature, that is to say, a degree $-2$ bracket that is again compatible (in a graded sense) with the degree zero structure, {\em viz.}, the cup product. More precisely,
in Theorem \ref{schnief} and Corollary \ref{schnief2}, we conclude by:

{\renewcommand{\thetheorem}{{D}}
\begin{theorem}
\label{D}
Let $H$ be a finite dimensional Hopf algebra over a field of characteristic zero. Then the Gerstenhaber bracket from Theorem \ref{B} vanishes on Gerstenhaber-Schack cohomology, which therefore becomes an $e_3$-algebra.
\end{theorem}
}

\subsection{Notation}
\label{notation}
In the following, $(H, m, \Delta, \eta, \gve, S)$ is most of the time a not necessarily finite dimensional Hopf algebra over a field $\K$ (sometimes merely a ring, most of the times of characteristic zero) with invertible antipode $S$ unless otherwise specified. 
Here, $m: H \otimes H \to H$ denotes the multiplication and $\eta: k \to H$ the unit in $H$, mostly suppressed in notation as is the specification of $\K$ in unadorned tensor products. For the coproduct, we use the customary Sweedler notation $\gD(u) = u_{(1)} \otimes u_{(2)}$ for $u \in H$. 
For two left $H$-modules $M$ and $N$, we will denote the left diagonal action on $M \otimes N$ by
\begin{equation}
\label{monoidal}
u \lact (m \otimes n) := u_{(1)} m \otimes u_{(2)} n
\end{equation}
for $m \in M$ and $n \in N$, and likewise the right one $(m \otimes n) \ract u := m u_{(1)}  \otimes n u_{(2)} $ if $M$ and $N$ were right $H$-modules instead.
The base field $\K$ will be considered as an $(H,H)$-bimodule via the the counit $\gve$:
\[
u\lact 1:= \gve(u)=:1\ract u.
\]
Furthermore, for the sake of compactness in notation, we often abbreviate a tensor power $u^1 \otimes \cdots \otimes u^n \in H^{\otimes n}$ by a tuple $(u^1, \ldots, u^n)$.
Quite on the contrary, for a $\K$-linear map $f: H^{\otimes p} \to H^{\otimes q}$, for all $p,q \geq 0$ we will sometimes use the notation
\begin{equation*}
\label{}
f^{(1)}(u_1, \ldots, u_p) \otimes \cdots \otimes f^{(q)}(u_1, \ldots, u_p) \in H^{\otimes q}
\end{equation*}
if we think that this increases comprehension.
Also, we often write $[a \otimes b][c \otimes d]$ to express the factorwise multiplication $ac \otimes bd$ on a tensor product.

\section{Gerstenhaber-Schack cohomology}

\subsection{The Gerstenhaber-Schack double and diagonal complex}

Whereas the formal deformation theory of an associative $\K$-algebra leads
to the Hochschild cochain complex, deforming a bialgebra $H$, that is, deforming simultaneously the multiplication and the comultiplication maintaining the compatibility between the two of them leads to a double complex with entries
$$
C^{pq}_{GS}(H,H) := \Hom_k(H^{\otimes p}, H^{\otimes q})
$$
for $p,q \geq 0$; see \cite[\S8]{GerSch:ABQGAAD} for all details as well as \cite[\S3]{Sho:TOABFATMC} for more information. The columns in this bicomplex  are given by the Hochschild cochain complex and Hochschild coboundary $\gd^v$ with values in the various tensor powers of $H$, whereas the rows are the coHochschild (also known as {\em Cartier}) cochain complexes with coHochschild coboundary $\delta^h$, equally with values in the various tensor powers of $H$:
\begin{equation}
\label{tauschetasche}
	\xymatrix{
\cdots & \cdots & \cdots &  
\\
\Hom_\K(H^{\otimes 2},\K) \ar_{\delta^v}[u] \ar^{\delta^h}[r] & \Hom_\K(H^{\otimes 2},H)  \ar_{\delta^v}[u] \ar^{\delta^h}[r] & \Hom_\K(H^{\otimes 2},H^{\otimes 2}) \ar_{\delta^v}[u] \ar[r] & \cdots 
\\
\Hom_\K(H,\K) \ar_{\delta^v}[u] \ar^{\delta^h}[r] & \Hom_\K(H,H) \ar_{\delta^v}[u] \ar^{\delta^h}[r] & \Hom_\K(H,H^{\otimes 2}) \ar_{\delta^v}[u] \ar[r] & \cdots 
\\
\Hom_\K(\K,\K) \ar_{\delta^v}[u] \ar^{\delta^h}[r] & \Hom_\K(\K,H) \ar_{\delta^v}[u] \ar^{\delta^h}[r] & \Hom_\K(\K,H^{\otimes 2}) \ar_{\delta^v}[u] \ar[r] & \cdots 
}
\end{equation}
Here, we start counting by zero, {\em i.e.}, the leftmost column and the bottom row will be referred to as the zeroth column and row, respectively.
The vertical resp.\ horizontal differential explicitly look as follows. For each column and $q \geq 0$, define the cosimplicial $\K$-module $C^\bullet(H, H^{\otimes q}) := \Hom_\K(H^{\otimes \bullet},H^{\otimes q})$ with the following cofaces and codegeneracies: for any $f \in C^n(H, H^{\otimes q})$,
set
\begin{equation}
\begin{array}{rcl}
\label{sosoaha}
(\gd^v_i f)(u^0, \ldots, u^n) \!\!\!\! &:=  \!\!\!\!&
\left\{  
\begin{array}{lcl}
u^0 \lact f(u^1, \ldots, u^n)
& \mbox{if} & i = 0, 
\\
 f(u^0, \ldots, u^{i-1} u^i, \ldots, u^n)
& \mbox{if} & 1 \leq i \leq  n, 
\\
 f(u^0, \ldots, u^{p-1}) \ract u^n
& \mbox{if} & i = n+1, 
\\
\end{array}\right.
\\
\
\\
(\gs^v_j f)(u^1, \ldots, u^{n-1})  \!\!\!\! &:=  \!\!\!\! &  
f(u^1, \ldots, u^{j}, 1, u^{j+1}, \ldots, u^{n-1})
, 
\quad  0 \leq j \leq n-1, 
\end{array}
\end{equation}
where $\lact, \ract$ denote the left resp.\ right diagonal action \rmref{monoidal}, and where we, as we will often do, denote tensor chains by tuples, {\em cf.} \S\ref{notation}.
Observe that when $q = 0$, the diagonal action on $\K$ is to be understood via the counit $\gve$. As usual, put $\delta^v: = \sum^{n+1}_{i=0} (-1)^i \delta^v_i$. 

Likewise, for each row and $p \geq 0$, define another cosimplicial $\K$-module by $C^\bullet(H^{\otimes p}, H) := \Hom_\K(H^{\otimes p},H^{\otimes \bullet})$ along with the following cofaces and codegeneracies: for any $f \in C^n(H^{\otimes p}, H)$,
set
\begin{small}
\begin{equation}
\begin{array}{rcl}
\label{susuaha}
(\gd^h_i f)(u^1, \ldots, u^p) \!\!\!\! &:=  \!\!\!\!&
\left\{  
\begin{array}{lcl}
u^1_{(1)} \cdots u^p_{(1)} \otimes  f(u^1_{(2)}, \ldots, u^p_{(2)})
& \mbox{if} & i = 0, 
\\
(\id^{i-1} \otimes \Delta \otimes \id^{q-i}) f(u^1, \ldots, u^p)
& \mbox{if} & 1 \leq i \leq  n, 
\\
f(u^1_{(1)}, \ldots, u^p_{(1)}) \otimes u^1_{(2)} \cdots u^p_{(2)} 
& \mbox{if} & i = n+1, 
\\
\end{array}\right.
\\
\
\\
(\gs^h_j f)(u^1, \ldots, u^{p})  \!\!\!\! &:=  \!\!\!\! &  
(\id^{j} \otimes \gve \otimes \id^{n-j-1}) f(u^1, \ldots, u^{p})
, 
\quad  0 \leq j \leq n-1. 
\end{array}
\end{equation}
\end{small}
Analogously as above, we put  $\delta^h: = \sum^{n+1}_{i=0} (-1)^i \delta^h_i$.

\begin{dfn}
The cohomology of the total complex associated to the Gerstenhaber-Schack bicomplex is called {\em Gerstenhaber-Schack cohomology} and denoted 
$$
H^\bullet_{GS}(H,H) := H^\bullet(\Tot C_{GS}(H,H)).
$$
\end{dfn}

It is a straightforward check that $C^{\bullet\bullet}_{GS}(H,H)$ becomes a bi-cosimplicial object by means of the quadruple $(\gd^v_\bull, \gs^v_\bull,  \gd^h_\bull, \gs^h_\bull)$;
hence, by the Dold-Puppe generalisation \cite{DolPup:HNAFA} of the Eilenberg-Zilber theorem, one has 
$$
H^\bullet_{GS}(H,H) := H^\bullet(\Tot C_{GS}(H,H)) \simeq H^\bullet(C_{\rm diag} (H)).
$$
Here, 
$$
C_{\rm diag}^n (H) := C^{nn}_{GS}(H,H)
$$ 
is the {\em diagonal (cochain) complex} with coboundary $\gd^\diag := \sum^{n+1}_{i=0} (-1)^i \delta^\diag_i$ with cofaces given as
\begin{equation}
\label{hihihaha}
\delta^\diag_i := \delta^v_i \circ \delta^h_i = \delta^h_i \circ \delta^v_i 
\end{equation}
for all $0 \leq i \leq n+1$ and codegeneracies by $\gs^\diag_j := \gs^v_j \circ \gs^h_j = \gs^h_j \circ \gs^v_j$ for $0 \leq j \leq n-1$.
In other words, $H^\bullet_{GS}(H,H)$ can be computed by the cosimplicial $\K$-module
\begin{equation}
\label{gazette}
C_{\rm diag}^n (H) = \Hom_k(H^{\otimes n}, H^{\otimes n}),
\end{equation}
where the cofaces and codegeneracies are explicitly given as follows:
\begin{equation}
\begin{split}
\label{trockenfirma}
(&\gd^\diag_i f)(u^0, \ldots, u^n) 
\\
&:=  
\left\{  
\begin{array}{lcl}
u^0_{(1)} \cdots u^n_{(1)} \otimes  u^0_{(2)}  \lact f(u^1_{(2)}, \ldots, u^n_{(2)})
& \mbox{if} & i = 0, 
\\
 (\id^{i-1} \otimes \Delta \otimes \id^{n-i}) f(u^0, \ldots, u^{i-1} u^i, \ldots, u^n) 
& \mbox{if} & 1 \leq i \leq  n, 
\\
f(u^0_{(1)}, \ldots, u^{n-1}_{(1)}) \ract u^n_{(1)}  \otimes u^1_{(2)} \cdots u^n_{(2)} 
& \mbox{if} & i = n+1, 
\\
\end{array}\right.
\\
& 
\
\\
(&\gs^\diag_j f)(u^1, \ldots, u^{n-1})  
\\
&:= 
(\id^{j} \otimes \gve \otimes \id^{n-j-1}) f(u^1, \ldots, u^{j}, 1, u^{j+1}, \ldots, u^{n-1})
, 
\quad  0 \leq j \leq n-1. 
\end{split}
\end{equation}

We will sometimes say {\em Gerstenhaber-Schack complex} to mean the diagonal complex just introduced, in contrast to {\em Gerstenhaber-Schack bicomplex} to mean the full bicomplex.

\section{Cyclic structures for the Gerstenhaber-Schack complex}
\label{steckdose}

In this section, we will show how to obtain a para-cocyclic structure both on the rows and on the columns of the Gerstenhaber-Schack bicomplex.

For a few details on (para-)cocyclic modules and terminology, see Appendix \ref{defilambda}.

\subsection{Para-cocyclic structures on the columns}
\label{yippie}

The idea of how to obtain a para-cocyclic structure on the columns 
goes as follows: when the Hopf algebra $H$ is merely seen as a $k$-algebra, that is, ignoring the full Hopf structure, the Hochschild cochain complex (the first column in the Gerstenhaber-Schack bicomplex) does not carry a cocyclic structure unless certain (Yetter-Drinfel'd type) conditions are fulfilled (see \cite[Thm.~1.3]{Kow:WEIABVA}) as is the case for, {\em e.g.}, Frobenius or Calabi-Yau algebras. 

On the other hand, using the full Hopf structure incorporated in the left adjoint action $H \otimes M \to M, \ u \otimes m \mapsto u_{(1)} m S(u_{(2)})$,
it is a well-known fact that for any $H$-bimodule $M$ there is a cochain isomorphism
\begin{equation}
\label{sesam0}
\Hom_k(H^{\otimes \bullet}, M) \simeq \Hom_k(H^{\otimes \bullet}, \ad(M))
\end{equation}
such that 
(via the mapping theorem VIII.3.1 in \cite{CarEil:HA})
the Hochschild cohomology of $H$ with values in the $H$-bimodule $M$ is isomorphic to its Hopf algebra cohomology with values in the left $H$-module $\ad(M)$, that is,
\begin{equation}
\label{sesam0a}
\Ext^\bullet_{H^e_\alg}
(H,M) \simeq \Ext^\bullet_{H}(k,\ad(M)).
\end{equation}
Explicitly, this map is defined as
\begin{equation}
\begin{array}{rcl}
\label{sesam1}
\xi: \Hom_k(H^{\otimes n}, M) &\to& \Hom_k(H^{\otimes n}, \ad(M)) \\
f &\mapsto& \{(u^1, \ldots, u^n) \mapsto f(u^1_{(1)}, \ldots, u^n_{(1)}) S(u^1_{(2)} \cdots u^n_{(2)}) \} 
\end{array}
\end{equation}
in degree $n$, where we again denoted tensor powers by tuples. It is easy to see that
\begin{equation*}
\label{sesam2}
g \mapsto \{  (u^1, \ldots, u^n) \mapsto g(u^1_{(1)}, \ldots, u^n_{(1)}) u^1_{(2)} \cdots u^n_{(2)}\}
\end{equation*}
gives the inverse. 
As a side remark, we add here that (at least in case $M= H$ or more generally if $M$ is a braided commutative Yetter-Drinfel'd algebra) this map even induces a map of Gerstenhaber algebras, see \cite[Ex.~3.5]{Kow:BVASOCAPB}.

On the other hand, by \cite[Thm.~1.2]{Kow:WEIABVA} the cochain complex defined by the right hand side in \rmref{sesam0} {\em can} be turned into a (para-)cocyclic $\K$-module provided that $M$ is a so-called {\em  anti Yetter-Drinfel'd (aYD)  contramodule} over $H$, see \S\ref{guertel} right below for more details. Hence, we will check for $M := H^{\otimes q}$ for $q \geq 0$ whether this is the case and if so, transfer the (para-)cocyclic structure from the right hand side in \rmref{sesam0} by means of the cochain isomorphism to its left hand side, that is, to the columns in the Gerstenhaber-Schack bicomplex.
Afterwards, in \S\ref{yuppue} we explain a dual theory the full details of which will appear elsewhere.

\subsubsection{Hopf algebra cohomology}

Define for any left $H$-module $M$ (with action simply denoted by juxtaposition) the complex $\big(C^\bullet(H,M), \delta_\emme \big)$ as follows: put
\begin{equation}
\label{earlyinthemorning}
C^n(H,M) := \Hom_k(H^{\otimes n}, M) 
\end{equation}
in degree $n \in \N$ and give it the structure of a cosimplicial $\K$-module $\big(C^\bullet(H,M), \delta_\bull, \gs_\bull \big)$ by means of the following cofaces and codegeneracies: for every $f \in C^n(H,M)$, set
\begin{equation}
\begin{array}{rcl}
\label{deepatnight2}
(\gd_i f)(u^0, \ldots, u^n) \!\!\!\! &:=  \!\!\!\!&
\left\{  
\begin{array}{lcl}
u^0 f(u^1, \cdots, u^n) & \mbox{for} & i = 0, 
\\
f(u^0, \cdots, u^{i-1} u^{i}, \cdots, u^n) & \mbox{for} & 1 \leq i \leq  n, 
\\
f(u^0, \cdots, u^{n-1}) \gve (u^n)  & \mbox{for} & i = n+1, 
\\
\end{array}\right.
\\
\
\\
(\gs_j f)(u^1, \ldots, u^{n-1})  \!\!\!\! &:=  \!\!\!\! &  
 f(u^1, \ldots, u^{j}, 1, u^{j+1}, \ldots, u^{n-1}), 
\quad  0 \leq j \leq n-1. 
\end{array}
\end{equation}
Defining as usual the differential by 
$
\delta_\emme := \sum^{n+1}_{i=0} (-1)^i \delta_i,
$
the complex $\big(C^\bullet(H,M), \delta_\emme \big)$ then computes the {\em Hopf algebra cohomology}
$\Ext^\bullet_{H_\alg}(\K,M)$ with values in $M$, where $\K$ is seen as a left $H$-module via the counit. The subscript $_\alg$ here indicates for better distinction (compared to \S\ref{yuppue}) that this is an $\Ext$ in $\hmod$, the category of left $H$-modules for the $\K$-algebra $H$.

\subsubsection{Cocyclic structures on Hopf algebra cohomology}
\label{guertel}

Following \cite[Cor.~4.13]{Kow:WEIABVA}, in order to define a cocyclic operator on 
$\big(C^\bullet(H,M), \delta_\bull, \gs_\bull \big)$,
we have to assume more structure 
on the coefficients $M$: beyond being a left $H$-module, we also need 
$M$ to be a right $H$-contramodule, that is, comes equipped with a map $\gamma: \Hom_\K(H,M) \to M$ subject to certain (contra)associativity and (contra)unitality conditions (see {\em op.~cit.}
for all details; in case $H$ is finitely generated as a module over $\K$, think of a right $H^*$-module).
As in \cite[Thm.~1.2]{Kow:WEIABVA}, a 
para-cocyclic operator on this complex is now given by
\begin{equation}
\label{deeperatnight2a}
(\tau f)(u^1, \ldots, u^n) = 
\gamma\Big(u^1_{(1)} f\big(u^2_{(1)}, \ldots, u^n_{(1)}, S(u^1_{(2)} u^2_{(2)} \cdots u^n_{(2)})(-) \big)\Big).
\end{equation}

The so-defined para-cocyclic $\K$-module 
$\big(C^\bullet(H,M), \delta_\bull, \gs_\bull, \tau\big)$
becomes {\em cyclic}, that is, additionally fulfils $\tau^{n+1} = \id$ in degree $n$ only if $M$ is a {\em stable aYD contramodule}, that is, obeys the following compatibility between left $H$-action and right $H$-contraaction,
\begin{equation}
\label{nawas1}
u (\gamma(f)) = \gamma \big(u_{(2)} f(S(u_{(3)})(-)u_{(1)}) \big), \qquad \forall u \in H, \ f \in \Hom_\K(H,M),
\end{equation}
as well as the stability condition,
\begin{equation}
\label{stablehalt}
\gamma((-)m)= m
\end{equation}
for all $m \in M$, where we denote 
$(-)m \colon u \mapsto um$ as a map in $ \Hom_\K(H,M)$.

For example, the trivial contraaction $\gamma: \Hom_\K(H,M) \to M, \ f \mapsto f(1)$, which we most of the times suppress in notation, obviously fulfils stability in the sense of \rmref{stablehalt}, but it is equally obvious that it is in general {\em not}  
compatible with the $H$-action in the sense of the aYD-condition \rmref{nawas1} {\em unless} $H$ is cocommutative (or unless the $H$-action on $M$ is itself trivial in the sense of being given by $u \otimes m \mapsto \gve(u)m$), and hence for the trivial contraaction we in general do not have that $\tau^{n+1} = \id$ in degree $n$. 

Nevertheless, we want to closer examine the case of the trivial contraaction: the para-cocyclic operator \rmref{deeperatnight2a} here simplifies to 
\begin{equation}
\label{deeperatnight2}
(\tau f)(u^1, \ldots, u^n) = 
u^1_{(1)} f\big(u^2_{(1)}, \ldots, u^n_{(1)}, S(u^1_{(2)} u^2_{(2)} \cdots u^n_{(2)})\big),
\end{equation}
and if the antipode $S$ is invertible, $\tau$ is equally invertible with inverse
\begin{small}
\begin{equation}
\label{deeperatnight3}
\begin{split}
(\tau^{-1} f)&(u^1, \ldots, u^n) 
= 
u^1_{(3)} \cdots u^{n-1}_{(3)} u^n_{(2)}  f\big(S^{-1}(u^1_{(2)}  \cdots u^{n-1}_{(2)} u^n_{(1)}), u^1_{(1)}, \ldots,  u^{n-1}_{(1)} \big),
\end{split}
\end{equation}
\end{small}
which is the operator we shall deal with in what follows and which is why we drop the superscript $^{-1}$ from the notation ({\em i.e.}, we tacitly exchange the r\^oles of $\tau$ and $\tau^{-1}$).

As our goal in this section is to find a para-cocyclic operator on the columns in the Gerstenhaber-Schack bicomplex, we will focus on the case in which $M$ is given by the tensor powers $H^{\otimes q}$ for $q \geq 0$ seen as a left $H$-module by means of the adjoint action
$$
u \rightslice (v^1 \otimes \cdots \otimes v^p) :=  u_{(1)} \lact (v^1 \otimes \cdots \otimes v^p) \ract S(u_{(2)}),
$$ 
where $\lact, \ract$ denote the left resp.\ right diagonal action as in \S\ref{notation}. We will refer to this situation by $\ad(H^{\otimes q})$.
With this, we form the cochain complex $C^\bullet(H,\ad(H^{\otimes q})) = \Hom_k(H^{\otimes \bullet}, \ad(H^{\otimes q}))$ for every $q \geq 0$ with differential
\begin{equation*}
\begin{split}
(\delta_\ad f)(u^0, \ldots, u^n) &= u^0_{(1)} \lact f(u^1, \ldots, u^n) \ract S(u^0_{(2)}) 
\\
&
\quad
+ \sum_{i=0}^{n-1} (-1)^{i+1} f(u^0, \cdots, u^i u^{i+1}, \cdots, u^n) 
\\
& 
\quad
+ (-1)^{n+1}  f(u^0, \ldots, u^{n-1}) \gve(u^n), 
\end{split}
\end{equation*}
which computes the groups $\Ext^\bullet_{H_\alg}(\K,{\rm ad}(H^{\otimes q}))$. As above 
\begin{equation}
\label{deeperatnight4}
\begin{split}
(\tau_\ad f)&(u^1, \ldots, u^n) 
\\
&
= 
(u^1_{(3)} \cdots u^{n-1}_{(3)} u^n_{(2)}) \rightslice  f\big(S^{-1}(u^1_{(2)}  \cdots u^{n-1}_{(2)} u^n_{(1)}), u^1_{(1)}, \ldots,  u^{n-1}_{(1)} \big)
\end{split}
\end{equation}
hence yields a para-cocyclic operator on (the cosimplicial $\K$-module associated to) $\big(C^\bullet(H,\ad(H)), \delta_\ad \big)$. 

\subsubsection{Transferring the para-cocyclic structure} 
\label{yeppee}


The cochain isomorphism \rmref{sesam1} for $M := H^{\otimes q}$, $q \geq 0$, explicitly looks as follows:  
\begin{eqnarray*}
C^n(H, H^{\otimes q}) &\to& C^n(H, \ad(H^{\otimes q})) 
\\
 \xi: f \mapsto \{(u^1, \ldots, u^n) &\mapsto& f(u^1_{(1)}, \ldots, u^n_{(1)}) \ract S(u^1_{(2)} \cdots u^n_{(2)}) \} \\
 \xi^{-1}: \{  g(u^1_{(1)}, \ldots, u^n_{(1)}) \ract (u^1_{(2)} \cdots u^n_{(2)})  &\mapsfrom& (u^1, \ldots, u^n) \} \mapsfrom g
\end{eqnarray*}
for all $q \geq 0$, where as before $\lact, \ract$ denote the left resp.\ right diagonal action on an element $f(u^1, \ldots, u^n) \in H^{\otimes q}$. Hence, Eq.~\rmref{sesam2} reads as
$$
\Ext^\bullet_{H^e_\alg}
(H,H^{\otimes q}) \simeq \Ext^\bullet_{H}(k,\ad(H^{\otimes q})).
$$ 
By means of this cochain isomorphism, we now transport the para-cocyclic operator $\tau_\ad$ from \rmref{deeperatnight4} to the para-cocyclic operator 
$$
\tau_\alg := \xi^{-1} \circ \tau_\ad \circ \xi
$$
on the Hochschild complex for $H$ with coefficients in $H^{\otimes q}$, that is, to the columns in the Gerstenhaber-Schack bicomplex \rmref{tauschetasche}. Explicitly, after a straightforward computation this comes out as
\begin{equation*}
\begin{split}
(\tau_\alg f)(u^1, \ldots, u^n) =  (u^1_{(3)} \cdots u^n_{(3)}) \lact  f\big(S^{-1}(u^1_{(2)} \cdots u^n_{(2)}), u^1_{(1)}, \ldots, u^{n-1}_{(1)}\big) \ract u^n_{(1)}.
\end{split}
\end{equation*}
A tedious but straightforward computation the details of which we also happily omit yields
\begin{equation}
\label{pausenscheibe}
\begin{split}
(\tau_\alg^{n+1} f)&(u^1, \ldots, u^n) 
\\
&
=  \big(u^1_{(5)} \cdots u^n_{(5)} S^{-1}(u^1_{(2)} \cdots u^n_{(2)})\big) \lact  f\big(S^{-2}(u^1_{(3)}), \ldots S^{-2}(u^n_{(3)})\big) 
\\
&
\quad
\ract  \big(S^{-1}(u^1_{(4)} \cdots u^n_{(4)}) u^1_{(1)} \cdots u^n_{(1)}\big),
\end{split}
\end{equation}
which is the identity if $H$ is cocommutative (which in turn implies $S^{-2} = \id$). Summing up, we obtain:

\begin{prop}
\label{columns}
For all $q \geq 0$, the cosimplicial $\K$-module associated to the Hochschild complex $\big(C^\bullet(H, H^{\otimes q}), \gd^v\big)$ for a Hopf algebra $H$ seen merely as a $\K$-algebra (that is, the columns in the Gerstenhaber-Schack bicomplex) can be given the structure of a para-cocyclic $\K$-module, which is cocyclic if $H$ is cocommutative.
\end{prop}

\subsection{Para-cocyclic structures on the rows}
\label{yuppue}

This section can be seen as a sort of dual theory to that developed in \S\ref{yippie} the underlying idea being very similar: dual to the above, 
let $M$ be an $H$-bicomodule and consider it also as a right $H$-comodule $\coad(M)$ by means of its right coadjoint coaction. Similarly to the above, 
one has a cochain isomorphism
\begin{equation}
\label{mases0}
\Hom_k(M, H^{\otimes \bullet}) \simeq \Hom_k({\rm coad}(M), H^{\otimes \bullet})
\end{equation}
between the coHochschild (or {\em Cartier}) cochain complex on the left hand side
and the (what we call since we did not find a better name for it) Hopf coalgebra cohomology
such that 
$$
\Ext^\bullet_{H^e_\coalg}
(H,M) \simeq \Ext^\bullet_{H}(k,{\rm coad}(M)),
$$
where $H^e_\coalg$ this time denotes the enveloping coalgebra. This isomorphism is (for certain coefficients $M$ such as $H$ itself) again one of Gerstenhaber algebras, which will be proven elsewhere. 
As above, we are mainly interested in the case $M := H^{\otimes q}$, $q \geq 0$, now seen as $H$-comodules, and here the isomorphism \rmref{mases0} comes out as
\begin{equation}
\label{mases1}
\begin{array}{rcl}
\Hom_k(H^{\otimes p}, H^{\otimes n}) &\to& \Hom_k({\rm coad}(H^{\otimes p}), H^{\otimes n}) \\
 \eta: f \mapsto \{(u^1, \ldots, u^p) &\mapsto& S(u^1_{(1)} \cdots u^p_{(1)}) \lact f(u^1_{(2)}, \ldots, u^p_{(2)}) \} \\
 \eta^{-1}: \{ (u^1_{(1)} \cdots u^p_{(1)}) \lact g(u^1_{(2)}, \ldots, u^p_{(2)}) &\mapsfrom& (u^1,  \ldots, u^p) \} \mapsfrom g.
\end{array}
\end{equation}

Although this dual theory is somewhat involved, we will be quite brief and defer most of the details to a separate publication, notably the comments on so-called {\em countermodules}, see below. 

\subsubsection{Hopf coalgebra cohomology}
Define for any right $H$-comodule $M$ (with coaction denoted by the usual Sweedler notation $M \to M \otimes H, \ m \mapsto m_{(0)} \otimes m_{(1)}$) the complex $\big(C^\bullet(M,H), \delta^\emme \big)$ as follows: put
$C^n(M,H) := \Hom_k(M, H^{\otimes n})$ in each degree $n \in \N$ 
and give it the structure of a cosimplicial $\K$-module $\big(C^\bullet(M,H), \delta_\bull, \gs_\bull \big)$ by means of the following cofaces and codegeneracies: for every $f \in C^n(M,N)$, set
\begin{equation}
\begin{array}{rcl}
\label{deepatnight6}
(\gd_i f)(m) \!\!\!\! &:=  \!\!\!\!&
\left\{  
\begin{array}{lcl}
 1 \otimes f(m) & \mbox{for} & i = 0, 
\\
 (\id^{i-1} \otimes \gD \otimes \id^{n-i}) f(m) & \mbox{for} & 1 \leq i \leq  n, 
\\
 f(m_{(0)}) \otimes m_{(1)}  & \mbox{for} & i = n+1, 
\\
\end{array}\right.
\\
\
\\
(\gs_j f)(m)  \!\!\!\! &:=  \!\!\!\! &  (\id^{j} \otimes \gve \otimes \id^{n-1-j}) f(m), \quad 0 \leq j \leq n-1.
\end{array}
\end{equation}
Defining again the differential as the sum over all pieces, that is, by 
$
\delta^\emme := \sum^{n+1}_{i=0} (-1)^i \delta_i,
$
the complex $\big(C^\bullet(M,H), \delta^\emme \big)$ then 
computes the {\em Hopf coalgebra cohomology}
$\Ext^\bullet_{H_\coalg}(\K,M)$ with values in $M$, where $\K$ is seen as a trivial right $H$-comodule via the unit in $H$.  The subscript $_\coalg$ here indicates that this is an $\Ext$ in $\comodh$, the category of right $H$-comodules.

\subsubsection{Cocyclic structures on Hopf coalgebra cohomology}

In \S\ref{guertel}, to obtain a cocyclic operator, we had to assume an extra structure on the left $H$-modules $M$ given by a right $H$-contraaction which can be seen as a map $\gamma_\emme: \Hom_\K(H,M) \to \Hom_\K(\K,M)$. In the picture at hand, one has to equip right $H$-comodules with a sort of dual contraaction, more precisely with a map 
$$
\gamma^\emme: \Hom_\K(M,H) \to \Hom_\K(M,\K), 
$$
again subject to a certain associativity condition (which, however, turns out to be more involved as the Hom-tensor adjunction only works in one direction). One might call these objects {\em countermodules}, but we refrain from spelling out the details here. Whereas contraactions can be thought of as right $H^*$-actions in case $H$ is finitely generated as a module over $\K$, a counteraction might be thought of as a right $H$-action on $\Hom_\K(M,\K) =: M^*$ in case $M$ is finitely generated over $\K$ instead.

Dual to \S\ref{guertel}, we will now focus on the case in which the coefficients $M$ are given by the right $H$-comodule $H^{\otimes p}$ for $p \geq 0$ by means of the coadjoint coaction
\begin{equation}
\label{schnellschnell}
\begin{array}{rcl}
H^{\otimes p} &\to& H^{\otimes p} \otimes H, 
\\
u^1 \otimes \cdots \otimes u^p &\mapsto&
  u^1_{(2)} \otimes \cdots \otimes u^p_{(2)} \otimes  S(u^1_{(1)} \cdots u^p_{(1)})  u^1_{(3)} \cdots u^p_{(3)},
\end{array}
\end{equation}
and we will denote  $H^{\otimes p}$ equipped with this right coaction by $\coad(H^{\otimes p})$. With this, we 
obtain the cochain complex $C^\bullet(\coad(H^{\otimes p}), H) = \Hom_k(\coad(H^{\otimes p}), H^{\otimes \bullet})$ for every $p \geq 0$ with differential
\begin{equation*}
\begin{split}
(\delta_\coad f)(u^1, \ldots, u^p) 
&
= 1 \otimes f(u^1, \ldots, u^p) 
\\
&
\qquad
+ \sum_{i=1}^{n} (-1)^{i} (\id^{i-1} \otimes \Delta \otimes \id^{n-i}) f(u^1, \ldots, u^p) 
\\
& 
\qquad
+ (-1)^{n+1}  f(u^1_{(2)}, \ldots, u^p_{(2)}) \otimes  S(u^1_{(1)} \cdots u^p_{(1)})  u^1_{(3)} \cdots u^p_{(3)}, 
\end{split}
\end{equation*}
for $f \in  C^n(\coad(H^{\otimes p}), H)$,
which computes the groups $\Ext^\bullet_{H_\coalg}(\K,{\rm coad}(H^{\otimes p}))$.

Adding now the trivial counteraction  $\gamma^\emme:  \Hom_\K(M,H) \to \Hom_\K(M,\K), \ f \mapsto \gve \circ f$ given by the counit, one obtains a cocyclic operator:

\begin{lem}
\label{coad}
In the situation above, 
for $f \in  C^n(\coad(H^{\otimes p}), H)$,
the operator
\begin{footnotesize}
\begin{equation}
\label{deeperatnight1047}
\begin{split}
&(\tau_\coad f)(u^1, \ldots, u^p) 
\\
&
= 
S\big(f^{(1)}(u^1_{(2)}, \ldots, u^p_{(2)})\big) \lact
\\
&
\quad
 \big( 
f^{(2)}(u^1_{(2)}, \ldots, u^p_{(2)}) \otimes \cdots \otimes f^{(n)}(u^1_{(2)}, \ldots, u^p_{(2)}) \otimes S(u^1_{(1)} \cdots u^p_{(1)})  u^1_{(3)} \cdots u^p_{(3)}
\big), 
\end{split}
\end{equation}
\end{footnotesize}
where as always $\lact$ denotes the left diagonal $H$-action,
yields a para-cocyclic operator that completes the cosimplicial $\K$-module associated to the cochain complex $\big(C^\bullet(\coad(H^{\otimes q}), H), \delta_\coad \big)$ to a para-cocyclic $\K$-module, which is cyclic if $H$ is cocommutative. 
\end{lem}

\begin{proof}
That  $\big(C^\bullet(\coad(H^{\otimes q}), H), \delta_\bull, \gs_\bull, \tau_\coad \big)$ with the cosimplicial pieces from \rmref{deepatnight6} and the coadjoint coaction \rmref{schnellschnell} gives a para-cocyclic $\K$-module is a nasty but straightforward computation which we omit. The last statement will follow (somewhat backwards) from Proposition \ref{rows}.
\end{proof}

\subsubsection{Transferring the para-cocyclic structure}
Analogously to \S\ref{yeppee}, we now transport the operator $\tau_\coad$ to the rows in the Gerstenhaber-Schack bicomplex by means of the cochain isomorphism \rmref{mases1}, that is,
\begin{equation}
\label{coalg}
\tau_\coalg := \eta^{-1} \circ \tau_\coad \circ \eta.
\end{equation}
Explicitly, we obtain the following para-cocyclic operator on the coHochschild (or Cartier) cochain complex $C^\bullet(H^{\otimes p}, H) = \Hom_k(H^{\otimes p}, H^{\otimes \bullet})$:
\begin{equation}
\label{watnweesznicke}
\begin{split}
&(\tau_\coalg f)(u^1, \ldots, u^p)
\\
&
=
\big(u^1_{(1)} \cdots u^p_{(1)} S(f^{(1)}(u^1_{(2)}, \ldots, u^p_{(2)}))\big) \lact
\\
&
\quad
 \big( 
f^{(2)}(u^1_{(2)}, \ldots, u^p_{(2)}) \otimes \cdots \otimes f^{(n)}(u^1_{(2)}, \ldots, u^p_{(2)}) \otimes u^1_{(3)} \cdots u^p_{(3)}
\big), 
\end{split}
\end{equation}
for every $f \in  C^n(H^{\otimes p}, H)$.

In total, and
analogously to Proposition \ref{columns}, we obtain:

\begin{prop}
\label{rows}
For all $p \geq 0$, the cosimplicial $\K$-module associated to the coHochschild (or Cartier) cochain complex $\big(C^\bullet(H^{\otimes p}, H), \gd^h\big)$ for a Hopf algebra $H$ seen merely as a $\K$-coalgebra (that is, the rows in the Gerstenhaber-Schack bicomplex) can be given the structure of a para-cocyclic $\K$-module, which is cyclic if $H$ is cocommutative.
\end{prop}

\begin{proof}
This is an immediate consequence of \rmref{coalg} and Lemma \ref{coad}. Since we did not show the statement about cyclicity and cocommutativity there (as we do not need an explicit expression for $\tau_\coad^{n+1}$ in what follows), we will do this here. One computes (after a while):
\begin{equation}
\label{schonwiederkrank}
\begin{split}
&(\tau_\coalg^{n+1} f)(u^1, \ldots, u^p)  
\\
&
=  \big(u^1_{(1)} \cdots u^p_{(1)} S(u^1_{(4)} \cdots u^p_{(4)})\big) \lact \big( (S^2 \otimes \cdots \otimes S^2) f(u^1_{(3)}, \ldots, u^p_{(3)}) \big) 
\\
&
\quad
\ract  \big(S(u^1_{(2)} \cdots u^p_{(2)}) u^1_{(5)} \cdots u^p_{(5)}\big),
\end{split}
\end{equation}
from which the last statement (as well as the one in Lemma \ref{coad}) is then obvious.
\end{proof}

\subsection{The cyclic structure on the diagonal}
\label{yappae}

By a {\em bi-para-cocyclic $\K$-module} we mean a double complex in which both the rows and the columns are para-cocyclic $\K$-modules, and where all vertical structure maps commute with all horizontal ones.

As in \cite[p.~164]{GetJon:TCHOCPA}, we call a {\em cylindrical} $\K$-module a bi-para-cocyclic $\K$-module in which for the respective vertical and horizontal cocyclic operators in degree $n$ additionally $\tau_v^{n+1} \circ \tau_h^{n+1} = \id$ is satisfied.

\begin{theorem}
\label{backspace}
For any Hopf algebra over a field with involutive antipode, the Gerstenhaber-Schack double complex defines a cylindrical $\K$-module. Therefore, 
the cocyclic operator 
$$
\tau_\diag  := \tau_\alg \circ \tau_\coalg = \tau_\coalg \circ \tau_\alg
$$ 
completes the cosimplicial $\K$-module $\big(C^\bullet_\diag(H), \gd^\diag, \gs^\diag \big)$ from \rmref{gazette} to a cocyclic $\K$-module. 
\end{theorem}

\begin{proof}
For an $f \in C^n_\diag(H)$, the cocyclic operator $\tau_\diag$ after a straightforward computation comes out as:
\begin{equation}
\begin{split}
\label{duschdas}
(&\tau_\diag f)(u^1, \ldots, u^n) 
\\&
= 
 \Big(u^1_{(1)} \cdots u^{n-1}_{(1)} S\big(f^{(1)}\big(S^{-1}(u^1_{(3)} \cdots  u^{n-1}_{(3)} u^n_{(n)}), u^1_{(2)}, \ldots, u^{n-1}_{(2)}\big)\big)\Big) \lact 
\\
&
\qquad
 \big(f^{(2)}\big(S^{-1}(u^1_{(3)} \cdots u^{n-1}_{(3)} u^n_{(n)}), u^1_{(2)}, \ldots, u^{n-1}_{(2)} \big) u^n_{(1)}  \otimes \cdots 
\\
&
\qquad 
\otimes (f^{(n)}\big(S^{-1}(u^1_{(3)} \cdots u^{n-1}_{(3)} u^n_{(n)}), u^1_{(2)}, \ldots, u^{n-1}_{(2)}  \big) u^n_{(n-1)} \otimes 1\big). 
\end{split}
\end{equation}
By the very definition of $\tau_\diag$, it is already clear from \S\S\ref{yippie}--\ref{yuppue} that $\big(C^\bullet_\diag(H), \gd^\diag, \gs^\diag \big)$ is a para-cocyclic $\K$-module; it remains to show that 
$$
\tau_\diag^{n+1} = \tau_\alg^{n+1} \circ \tau_\coalg^{n+1}  = \tau_\coalg^{n+1} \circ \tau_\alg^{n+1} =  \id
$$ 
in degree $n$: using the expressions \rmref{pausenscheibe} along with \rmref{schonwiederkrank}, we compute for an $f \in C^n_\diag(H)$:
\begin{footnotesize}
\begin{equation*}
\begin{split}
& (\tau_\coalg^{n+1} \tau_\alg^{n+1} f)(u^1, \ldots, u^n)
\\
&
=  \big(u^1_{(1)} \cdots u^p_{(1)} S(u^1_{(4)} \cdots u^p_{(4)})\big) \lact 
\\
&
\qquad
\big( (S^2 \otimes \cdots \otimes S^2) (\tau_\alg^{n+1} f)(u^1_{(3)}, \ldots, u^p_{(3)}) \big) 
\ract  \big(S(u^1_{(2)} \cdots u^p_{(2)}) u^1_{(5)} \cdots u^p_{(5)}\big)
\\
&
=  
\big(u^1_{(1)} \cdots u^p_{(1)} S(u^1_{(8)} \cdots u^p_{(8)})\big) \lact 
\\
&
\qquad
\Big( (S^2 \otimes \cdots \otimes S^2) 
\Big(
 \big(u^1_{(7)} \cdots u^n_{(7)} S^{-1}(u^1_{(4)} \cdots u^n_{(4)})\big) \lact  f\big(S^{-2}(u^1_{(5)}), \ldots S^{-2}(u^n_{(5)})\big) 
\\
&
\qquad
\ract  \big(S^{-1}(u^1_{(6)} \cdots u^n_{(6)}) u^1_{(3)} \cdots u^n_{(3)}\big)
\Big)
\Big) 
\ract  \big(S(u^1_{(2)} \cdots u^p_{(2)}) u^1_{(9)} \cdots u^p_{(9)}\big)
\\
&
=  
\Big(u^1_{(1)} \cdots u^p_{(1)} S(u^1_{(8)} \cdots u^p_{(8)}) S^2(u^1_{(7)} \cdots u^n_{(7)}) S(u^1_{(4)} \cdots u^n_{(4)}) 
\Big)
\lact
\\
&
\qquad  
\Big((S^2 \otimes \cdots \otimes S^2) 
f\big(S^{-2}(u^1_{(5)}), \ldots S^{-2}(u^n_{(5)})\big)
\Big) 
\\
&
\qquad
\ract  \Big(S(u^1_{(6)} \cdots u^n_{(6)}) S^2(u^1_{(3)} \cdots u^n_{(3)}) S(u^1_{(2)} \cdots u^p_{(2)}) u^1_{(9)} \cdots u^p_{(9)}\Big)
\\
&
=  
(S^2 \otimes \cdots \otimes S^2) 
f\big(S^{-2}(u^1), \ldots S^{-2}(u^n)\big), 
\end{split}
\end{equation*}
\end{footnotesize}
where we used $S^2(u_{(2)}) S(u_{(1)}) = S(u_{(2)}) S^2(u_{(1)}) = \gve(u)$. If the antipode is involutive, the above expression obviously gives the identity on $C^n_\diag(H)$ and hence 
$
\tau_\diag^{n+1} =  \id.
$
\end{proof}

\begin{rem}
We want to underline that both $\tau_\alg$ and $\tau_\coalg$ were only para-cocyclic {\em unless} $H$ is cocommutative, but this restriction somewhat surprisingly disappears on the diagonal, which is precisely characteristic for a cylindrical $\K$-module. 
The condition of the antipode to be involutive can be replaced by the notion of a {\em modular pair in involution} in the sense of \cite{ConMos:CCAHA}, but we shall not pursue this generalisation here as this is not the main focus of our present considerations.
\end{rem}

\begin{rem}[{\bf Cyclic Gerstenhaber-Schack cohomology}]
\label{cyclicgs}
The cyclic operator $\tau_\diag$ yields as in \rmref{e-mantra} the {\em cyclic} (or {\em Connes'}) {\em boundary} 
$B =: B_\diag$.
In particular, if $S^2 = \id$, then one has $B^2_\diag =0$ and the triple $(C^\bullet_\diag(H), \gd^\diag, B_\diag)$ becomes a mixed complex. By the cyclic Eilenberg-Zilber theorem \cite[Thm.\ 3.1]{GetJon:TCHOCPA}, we know that there is a quasi-isomorphism of ``para-chain'' complexes between the total complex and the diagonal one (see {\em op.~cit.}~for all details), and it makes therefore sense to consider the cohomology of the complex 
$$
(C^\bullet_\diag(H)[[u]], \gd^\diag + u B_\diag), 
$$
where $u$ is a degree $+2$ variable, and call it {\em cyclic Gerstenhaber-Schack cohomology}.
\end{rem}

\begin{rem}
The cocyclic structure on the Gerstenhaber-Schack complex in Theorem \ref{backspace} will later on also automatically follow from Corollary \ref{forwardspace} implied by the existence of the structure of a cyclic operad on the family of $\K$-modules underlying the diagonal complex. 
\end{rem}

\section{Higher structures on the Gerstenhaber-Schack complex}
\label{waldkonzert}

In this section, we will approach the ($\K$-modules underlying the) Gerstenhaber-Schack complex from a slightly different point of view, {\em i.e.}, by exhibiting its operadic resp.\ cyclic operadic structure.

\subsection{The Gerstenhaber-Schack complex as an operad with multiplication}

One might be tempted to think that the operadic structure on the  ($\K$-modules underlying the) Gerstenhaber-Schack complex can be obtained by somehow composing and generalising the classical ones for the endomorphism resp.\ coendomorphism operad of the first column resp.\ row in \rmref{tauschetasche}, but things appear to be much more intricate involving the antipode, that is, probably would not work for merely a $\K$-bialgebra.

\begin{theorem}
\label{japonvuduciel1}
The collection $\{\Hom_\K(H^{\otimes n}, H^{\otimes n})\}_{n \geq 0}$ of $\K$-modules carries the structure of an operad with multiplication whose associated cochain complex is the Gerstenhaber-Schack complex.
\end{theorem}

\begin{proof}
Recall first the notation for the (left or right) diagonal action $\lact$ resp.\ $\ract$ from \rmref{monoidal}. As in \S\ref{notation}, we also write $[a \otimes b][c \otimes d]$ to express the factorwise multiplication $ac \otimes bd$.

Define then the partial operadic composition for 
$f \in \Hom_\K(H^{\otimes p}, H^{\otimes p})$
and 
$g \in \Hom_\K(H^{\otimes q}, H^{\otimes q})$
as follows:
\begin{small}
\begin{equation}
\label{aglio}
\begin{split}
&(f \circ_i g)(u^1, \ldots, u^{p+q-1}) 
\\
&:= [(\id^{i-1} \otimes \Delta^{q-1} \otimes \id^{p-i}) f(u^1, \ldots, u^{i-1}, u^i_{(i+2)} \cdots  u^{i+q-1}_{(i+2)}, u^{i+q}_{(i+2)}, \ldots,  u^{p+q-1}_{(i+2)})]
\\
& 
\qquad
[S^{-1}(u^{i}_{(i+1)} \cdots u^{p+q-1}_{(i+1)}) \lact (u^{i}_{(1)} \cdots u^{p+q-1}_{(1)} \otimes \cdots \otimes u^{i}_{(i-1)} \cdots u^{p+q-1}_{(i-1)} 
\\
& 
\qquad 
\otimes g(u^{i}_{(i)}, \ldots, u^{i+q-1}_{(i)}) \ract (u^{i+q}_{(i)} \cdots u^{p+q-1}_{(i)})) \otimes 1^{\otimes p-i}],
\end{split}
\end{equation}
\end{small}
where $\Delta^0 := \id$, $\Delta^1 := \Delta$, $\Delta^2 := (\Delta \otimes \id)\Delta =  (\id \otimes \Delta)\Delta$, and so on.
If $g \in \Hom_\K(\K, \K) \simeq \K$
is an element of degree zero, that is, an element in $\K$, this formula has to be read by putting $\Delta^{-1} := \gve$, or rather
\begin{equation*}
\begin{split}
&(f \circ_i g)(u^1, \ldots, u^{p-1}) 
= (\id^{i-1} \otimes \gve \otimes \id^{p-i}) f(u^1, \ldots, u^{i-1}, \eta(g), u^{i}, \ldots,  u^{p-1}).
\end{split}
\end{equation*}
To check that this indeed fulfils the associativity axioms of the partial composition of an operad is an ordeal but nevertheless straightforward by enhanced Hopf algebra yoga: in order to avoid panic attacks among our readership, the full proof is moved to Appendix \S\ref{panic}, and we exemplify this here by a low degree computation: for example, let $f\in \Hom_\K(H^{\otimes 2}, H^{\otimes 2})$, $ g \in \Hom_\K(H^{\otimes 3}, H^{\otimes 3})$, and $h \in \Hom_\K(H, H)$, as well as $i= 2$, $j=3$. Then for $u, v,w, x \in H$, one computes

\begin{footnotesize}
\begin{equation*}
\begin{split}
\big(&(f \circ_2 g) \circ_3 h)(u, v, w, x) 
\\
&= [(f \circ_2 g)(u, v, w_{(5)}, x_{(5)})]
[S^{-1}( w_{(4)} x_{(4)}) \lact (w_{(1)} x_{(1)} \otimes  w_{(2)} x_{(2)}
\otimes h( w_{(3)}) \ract x_{(3)}) \otimes 1^{\otimes 2}]
\\
&=
 \big[[(\id \otimes \Delta^{2}) f(u, v_{(4)} w_{(8)} x_{(8)})]
[S^{-1}(v_{(3)} w_{(7)} x_{(7)}) \lact 
(v_{(1)} w_{(5)} x_{(5)}
\otimes g(v_{(2)}, w_{(6)}, x_{(6)}))]
\big]
\\
&
\qquad
\big[S^{-1}(w_{(4)} x_{(4)}) \lact 
(w_{(1)} x_{(1)} \otimes w_{(2)} x_{(2)}  
\otimes h(w_{(3)}) 
\ract x_{(3)}) \otimes 1\big]
\\
&=
\big[\big(\id \otimes \Delta^{2}\big)  
f(u, v_{(4)} w_{(7)} x_{(7)}) \big]
\big[S^{-1}(v_{(3)} w_{(6)} x_{(6)}) \lact (
v_{(1)} \otimes g(v_{(2)}, w_{(5)}, 
x_{(5)}))\big] 
\\
&
\qquad
\big[w_{(1)} x_{(1)} \otimes 
S^{-1}(w_{(4)} x_{(4)}) \lact
(w_{(2)} x_{(2)} \otimes h(w_{(3)}) \ract x_{(3)}) \otimes 1\big]
\\
&=
\big[\big(\id \otimes \Delta^{2}\big)  
f(u, v_{(4)} w_{(7)} x_{(7)}) \big]\big[S^{-1}(v_{(3)} w_{(6)} x_{(6)}) \lact
\\
&
\qquad 
\Big(v_{(1)} w_{(1)}  x_{(1)} \otimes 
[g(v_{(2)}, w_{(5)}, x_{(5)}) 
\ract (S^{-1}(w_{(4)}x_{(4)}))]
[w_{(2)} x_{(2)} \otimes h(w_{(3)}) 
\ract x_{(3)} \otimes 1]\Big)\big].
\end{split}
\end{equation*}
\end{footnotesize}
On the other hand, we have
\begin{footnotesize}
\begin{equation*}
\begin{split}
\big(&f \circ_2 (g \circ_2 h))(u,v,w,x)
\\
&= [(\id \otimes \Delta^{2}) f(u, v_{(4)} w_{(4)} x_{(4)})]
[S^{-1}(v_{(3)} w_{(3)} x_{(3)}) \lact (v_{(1)} w_{(1)} x_{(1)} \otimes (g \circ_{2} h)(v_{(2)}, w_{(2)}, x_{(2)}))]
\\
&=
\big[(\id \otimes \Delta^{2})f(u, v_{(4)} w_{(7)} x_{(7)})\big]
\big[S^{-1}(v_{(3)} w_{(6)} x_{(6)}) \lact 
\\
&
\qquad
\Big(v_{(1)} w_{(1)}x_{(1)} \otimes 
[g(v_{(2)}, w_{(5)}, x_{(5)})]
[S^{-1}(w_{(4)} x_{(4)}) \lact (w_{(2)} x_{(2)}
\otimes h(w_{(3)}) \ract x_{(3)}) \otimes 1]
\Big)\big]
\\
&=
\big[\big(\id \otimes \Delta^{2}\big)  
f(u, v_{(4)} w_{(7)} x_{(7)}) \big]\big[S^{-1}(v_{(3)} w_{(6)} x_{(6)}) \lact
\\
&
\qquad 
\Big(v_{(1)} w_{(1)}  x_{(1)} \otimes 
[g(v_{(2)}, w_{(5)}, x_{(5)}) 
\ract (S^{-1}(w_{(4)}x_{(4)}))]
[w_{(2)} x_{(2)} \otimes h(w_{(3)}) 
\ract x_{(3)} \otimes 1]\Big)\big],
\end{split}
\end{equation*}
\end{footnotesize}
which is the same expression as above.

 To see that we moreover deal with an operad with multiplication, define 
the multiplication element $ \mu \in \Hom_\K(H^{\otimes 2}, H^{\otimes 2})$
by
\begin{equation}
\label{stabiloboss}
\mu(u,v) = \Delta(uv) = u_{(1)}v_{(1)} \otimes u_{(2)}v_{(2)},
\end{equation}
along with the identity $\mathbb{1} := \id_H \in \Hom_\K(H, H)$ and the unit $e := 1_\K \in \Hom_\K(\K, \K) \simeq \K$. From the counitality of the underlying coalgebra, one immediately obtains $\mu \circ_1 e = \mu \circ_2 e = \mathbb{1}$ and likewise, for $u,v,w \in H$, we easily compute:
\begin{equation*}
\begin{split}
(&\mu \circ_1 \mu)(u,v,w)  
\\
&= 
[(\Delta \otimes \id) \mu(u_{(3)} v_{(3)}, w_{(3)})][S^{-1}(u_{(2)} v_{(2)}w_{(2)}) \lact  \mu(u_{(1)}, v_{(1)})  \ract w_{(1)}
\otimes 1]
\\
&=  
[(\Delta \otimes \id) \mu(u v, w)]
[1 \otimes 1 \otimes 1]
\\
&= u_{(1)} v_{(1)}w_{(1)} \otimes u_{(2)} v_{(2)}w_{(2)} \otimes u_{(3)} v_{(3)}w_{(3)}
\\
&=
[(\id \otimes \Delta) \mu(u, vw)]
[1 \otimes 1 \otimes 1]
%
\\
&
= (\mu \circ_2 \mu)(u,v,w).
\end{split}
\end{equation*}

As for the second statement, it is a straightforward check that the differential \rmref{hihihaha} on the diagonal complex is indeed as in 
\rmref{immaginedellacitta`} given by
$$
\delta^\diag f = (-1)^{p+1} \{\mu, f\},
$$
for $f \in \Hom_\K(H^{\otimes p}, H^{\otimes p})$,
where $\{\cdot, \cdot\}$ is defined as in \rmref{naemlichhier}. Since to see this nearly the same computation is required as the one we shall execute in the proof of Lemma \ref{haltnull}, we skip it at this point.
This completes the proof.
\end{proof}

\begin{rem}
One can probably relax the condition of the antipode to be invertible and replace the operadic composition \rmref{aglio} by its opposite which would use the antipode itself instead of its inverse. We do not pursue this generalisation here as in the next section we are mainly interested in the case of an involutive antipode.
\end{rem}

The following is then automatic by a well-known result, see \S\ref{pamukkale1} and the references given there.

\begin{cor}
\label{fegato}
For any Hopf algebra with invertible antipode, its Gerstenhaber-Schack cohomology is a Gerstenhaber algebra.
\end{cor}

Explicitly, the cup product and the Gerstenhaber bracket are given by Eqs.~\rmref{nuvole} and \rmref{naemlichhier}, respectively, with the operadic composition of Eq.~\rmref{aglio}; see also Lemma \ref{haltnull} below.

\begin{rem}
The question of the existence of a (nontrivial) Gerstenhaber algebra structure on Gerstenhaber-Schack cohomology has been asked in the literature several times before \cite{FarSol:GSOTCOHA, GerSch:ABQGAAD, Tai:IHBCOIDHAAGCOTYP}. In \S\ref{finite}, we shall show that if $H$ is finitely generated as a module over $\K$, then the Gerstenhaber bracket from Corollary \ref{fegato} vanishes in cohomology and one rather obtains an $e_3$-algebra structure, that is, a bracket of degree $-2$ which is compatible with the cup product (see, {\em e.g.}, \cite[Def.~5.1]{SalWah:FDOABVA} for a definition).
\end{rem}

\begin{example}[Universal enveloping algebras of Lie algebras]\label{pranzo-di-natale}
A simple computation for the enveloping algebra of a Lie algebra in, say, degree one, shows that the Gerstenhaber bracket from Corollary \ref{fegato} on Gerstenhaber-Schack cohomology does {\em not} vanish, see \rmref{dayne}. 
However, if the underlying Lie algebra is abelian, that is, the enveloping algebra both commutative and cocommutative, then the bracket vanishes again. 
\end{example}

\begin{rem}
In \S\ref{klempnerimmernochda} we are going to show that this Gerstenhaber structure is actually part of the stronger notion of a Batalin-Vilkoviski\u\i\ (or BV) structure.
\end{rem}

\begin{lem}
\label{haltnull}
The cup product is explicitly given by the formula 
\begin{equation}
\label{herbst}
\begin{split}
(&f \smallsmile g)(u^1,\ldots, u^{p+q}) 
\\
&
= f(u^1_{(1)}, \ldots, u^p_{(1)}) \ract (u^{p+1}_{(1)} \cdots u^{p+q}_{(1)}) 
\otimes (u^1_{(2)} \cdots u^p_{(2)}) \lact g(u^{p+1}_{(2)}, \ldots, u^{p+q}_{(2)}),
\end{split}
\end{equation}
for $f \in \Hom_\K(H^{\otimes p}, H^{\otimes p})$ 
and $g \in  \Hom_\K(H^{\otimes q}, H^{\otimes q})$.
\end{lem}

\begin{proof}
The cup product that belongs to the Gerstenhaber structure arising from an operad with multiplication is computed by the customary formula
$$
f \smallsmile g = (\mu \circ_2 g) \circ_1 f.
$$
Using \rmref{aglio} and the expression \rmref{stabiloboss} for the multiplication element $\mu$, we see by multiple coassociativity and the antipode identities for a Hopf algebra that
\begin{footnotesize}
\begin{equation*}
\begin{split}
(&(\mu \circ_2 g) \circ_1 f)(u^1, \ldots, u^{p+q}) 
\\
&= 
[(\Delta^{p-1} \otimes \id^{q}) \big((\mu \circ_2 g)(u^1_{(3)} \cdots  u^{p}_{(3)}, u^{p+1}_{(3)}, \ldots,  u^{p+q}_{(3)})\big)]
\\
& 
\qquad
[S^{-1}(u^1_{(2)} \cdots u^{p+q}_{(2)}) \lact f(u^{1}_{(1)}, \ldots, u^{p}_{(1)}) \ract (u^{p+1}_{(1)} \cdots u^{p+q}_{(1)}) \otimes 1^{\otimes q}]
\\
&= 
[(\Delta^{p-1} \otimes \id^{q}) 
\big(
[(\id \otimes \Delta^{q-1}) \mu(u^1_{(3)} \cdots  u^{p}_{(3)}, u^{p+1}_{(6)} \cdots  u^{p+q}_{(6)})]
\\
&
\qquad
[S^{-1}( u^{p+1}_{(5)} \cdots  u^{p+q}_{(5)}) \lact 
(u^{p+1}_{(3)} \cdots  u^{p+q}_{(3)} \otimes g( u^{p+1}_{(4)}, \ldots,  u^{p+q}_{(4)}))]\big)]
\\
& 
\qquad
[S^{-1}(u^1_{(2)} \cdots u^{p+q}_{(2)}) \lact f(u^{1}_{(1)}, \ldots, u^{p}_{(1)}) \ract (u^{p+1}_{(1)} \cdots u^{p+q}_{(1)}) \otimes 1^{\otimes q}]
\\
&= 
[u^1_{(3)} \cdots  u^{p+q}_{(3)} \otimes \cdots \otimes u^1_{(p+2)} \cdots  u^{p+q}_{(p+2)}  \otimes (u^1_{(p+3)} \cdots  u^{p}_{(p+3)}) 
\lact  g( u^{p+1}_{(p+3)}, \ldots,  u^{p+q}_{(p+3)})]
\\
& 
\qquad
[S^{-1}(u^1_{(2)} \cdots u^{p+q}_{(2)}) \lact f(u^{1}_{(1)}, \ldots, u^{p}_{(1)}) \ract (u^{p+1}_{(1)} \cdots u^{p+q}_{(1)}) \otimes 1^{\otimes q}]
\\
&= 
f(u^{1}_{(1)}, \ldots, u^{p}_{(1)}) \ract (u^{p+1}_{(1)} \cdots u^{p+q}_{(1)}) \otimes (u^1_{(2)} \cdots  u^{p}_{(2)}) 
\lact  g( u^{p+1}_{(2)}, \ldots,  u^{p+q}_{(2)}),
\end{split}
\end{equation*}
\end{footnotesize}
which is the desired formula.
\end{proof}

\begin{rem}
This formula coincides with the cup product given by Taillefer \cite[Prop.\ 4.1 \& Rem.\ 4.2]{Tai:IHBCOIDHAAGCOTYP} on the total complex when restricted to the diagonal.
\end{rem}

\subsection{The Gerstenhaber-Schack complex as a cyclic operad with multiplication}
\label{klempnerimmernochda}

Since we showed in the preceding section that the cosimplicial structure on the family $\{\Hom_\K(H^{\otimes n}, H^{\otimes n})\}_{n \geq 0}$ of $\K$-modules induced by its operadic composition \rmref{aglio} coincides with the one of the diagonal complex described in \rmref{trockenfirma}, we will from now on also use the operadic notation
$$
C_\diag(H)(n) := \Hom_\K(H^{\otimes n}, H^{\otimes n}), \qquad \forall n \geq 0.
$$
The aim of this section consists in showing that this operad with multiplication is in fact a cyclic operad with multiplication by means of the 
cyclic operator $\tau_\diag$ introduced in \rmref{duschdas}. As a consequence, we have by a (by now) well-known theorem that the Gerstenhaber algebra structure on cohomology mentioned in Corollary \ref{fegato} is in particular Batalin-Vilkoviski\u\i, that is, fulfils the equation
\begin{equation*}
\{f, g\} = - (-1)^f B f \smallsmile g - f \smallsmile B g + (-1)^f B(f \smallsmile g),
\end{equation*}
for $f,g \in C_\diag(H)(p)$, 
see Appendix \ref{pamukkale2}.

\begin{theorem}
\label{japonvuduciel2}
For any Hopf algebra with involutive antipode, the diagonal cochain complex constitutes a cyclic operad with multiplication. 
\end{theorem}

\begin{proof}
To check that the cyclic operator \rmref{duschdas} is defining a cyclic operad with multiplication w.r.t.\ the operadic composition in \rmref{aglio} is a straightforward computation without any tricks, but nasty. 
As before, being merciful with our readers, we will move the full proof to Appendix \ref{panic} and verify at this point only a nontrivial example in low degrees. More precisely, in order to illustrate \rmref{cycl2}, let  $f \in C_\diag(H)(2)$ and $g \in C_\diag(H)(2)$ and note first that for $u,v,w \in H$, one has
\begin{footnotesize}
\begin{equation*}
\begin{split}
&S\big((f \circ_1 g)^{(1)}(u, v, w)\big) \lact \big( 
(f \circ_1 g)^{(2)}(u, v,w) \otimes (f \circ_1 g)^{(3)}(u, v, w)\big) 
\\
&=
S\big(g^{(1)}(u_{(1)}, v_{(1)})  w_{(1)} \big) \lact 
\\
&
\qquad
\Big( 
g^{(2)}(u_{(1)}, v_{(1)}) w_{(2)} 
\otimes \big(u_{(2)} v_{(2)} w_{(3)} S\big(f^{(1)}(u_{(3)} v_{(3)}, w_{(4)})\big)\big) \lact f^{(2)}(u_{(3)} v_{(3)}, w_{(4)})\Big).
\end{split}
\end{equation*}
\end{footnotesize}
Using this, one obtains
\begin{footnotesize}
\begin{equation*}
\begin{split}
(& \tau_\diag (f \circ_1 g))(u, v, w) 
\\
&= 
 \Big(u_{(1)} v_{(1)} S\big((f \circ_1 g)^{(1)}\big(S^{-1}(u_{(3)} v_{(3)} w_{(3)}), u_{(2)}, v_{(2)}\big)\big)\Big) \lact 
\\
&
\qquad
 \big((f \circ_1 g)^{(2)}\big(S^{-1}(u_{(3)} v_{(3)} w_{(3)}), u_{(2)}, v_{(2)} \big) w_{(1)}  
\\
&
\qquad
\otimes (f \circ_1 g)^{(3)}\big(S^{-1}(u_{(3)} v_{(3)} w_{(3)}), u_{(2)}, v_{(2)}  \big) w_{(2)} \otimes 1\big) 
\\
& 
=
\Big(u_{(1)} v_{(1)}S(v_{(2)})  
S\big(g^{(1)}(S^{-1}(u_{(5)} v_{(6)} w_{(3)})_{(1)}, u_{(2)})\big)\Big) \lact 
\Big( 
g^{(2)}(S^{-1}(u_{(5)} v_{(6)} w_{(3)})_{(1)}, u_{(2)}) v_{(3)} w_{(1)}  
\\
&
\qquad
\otimes \big(
S^{-1}(u_{(5)} v_{(6)}  w_{(3)})_{(2)} u_{(3)} v_{(4)} S\big(f^{(1)}(S^{-1}(u_{(5)} v_{(6)}  w_{(3)})_{(3)} u_{(4)}, v_{(5)})\big)\big) \lact 
\\
&
\qquad
\big( 
f^{(2)}(S^{-1}(u_{(5)} v_{(6)}  w_{(3)})_{(3)} u_{(4)}, v_{(5)})w_{(2)}  
\otimes 1 \big)
\Big)
\\
&=
 \Big(u_{(1)} S\big(g^{(1)}(S^{-1}(u_{(3)} v_{(6)} w_{(3)}), u_{(2)}\big)\Big) \lact 
\\
&
\qquad
\Big( 
g^{(2)}(S^{-1}(u_{(3)} v_{(6)} w_{(3)}), u_{(2)}) v_{(1)} w_{(1)}  
\otimes \big(
S^{-1}(v_{(5)} w_{(4)}) v_{(2)} S\big(f^{(1)}(S^{-1}(v_{(4)} w_{(3)}), v_{(3)})\big)\big) \lact 
\\
&
\qquad
\big( 
f^{(2)}(S^{-1}(v_{(4)} w_{(3)}),  v_{(3)})w_{(2)} \otimes 1 \big)
\Big).
\end{split}
\end{equation*}
\end{footnotesize}
On the other hand, 
\begin{footnotesize}
\begin{equation*}
\begin{split}
(& \tau_\diag g \circ_2 \tau_\diag f)(u, v, w) 
\\
&= [(\id \otimes \Delta) \tau_\diag g(u, v_{(4)} w_{(4)})]
[S^{-1}(v_{(3)} w_{(3)}) \lact (v_{(1)} w_{(1)} \otimes \cdots \otimes u^{q}_{(q-1)} \otimes \tau_\diag f(v_{(2)}, w_{(2)}))]
\\
&= [(\id \otimes \Delta) 
\Big(
 \Big(u_{(1)} S\big(g^{(1)}\big(S^{-1}(u_{(3)} v_{(7)} w_{(6)}), u_{(2)}\big)\big)\Big) \lact 
\\
&
\qquad
 \big(g^{(2)}
\big(S^{-1}(u_{(3)} v_{(7)} w_{(6)}), u_{(2)}\big)
   v_{(6)} w_{(5)} \otimes 1\big) 
\Big)
]
[S^{-1}(v_{(5)} w_{(4)}) \lact 
\\
&
\qquad
\Big(v_{(1)} w_{(1)} \otimes 
\Big(v_{(2)} S\big(f^{(1)}\big(S^{-1}(v_{(4)} w_{(3)}), v_{(3)}\big)\big)\Big) \lact 
 \big(f^{(2)}\big(S^{-1}(v_{(4)} w_{(3)}), v_{(3)}\big)
w_{(2)} \otimes 1\big)\Big) 
]
\\
&=
 \Big(u_{(1)} S\big(g^{(1)}(S^{-1}(u_{(3)} v_{(6)} w_{(3)}), u_{(2)}\big)\Big) \lact 
\\
&
\qquad
\Big( 
g^{(2)}(S^{-1}(u_{(3)} v_{(6)} w_{(3)}), u_{(2)}) v_{(1)} w_{(1)}  
\otimes \big(
S^{-1}(v_{(5)} w_{(4)}) v_{(2)} S\big(f^{(1)}(S^{-1}(v_{(4)} w_{(3)}), v_{(3)})\big)\big) \lact 
\\
&
\qquad
\big( 
f^{(2)}(S^{-1}(v_{(4)} w_{(3)}),  v_{(3)})w_{(2)} \otimes 1 \big)
\Big),
\end{split}
\end{equation*}
\end{footnotesize}
that is, the same expression as above.

Checking (or rather illustrating at this point) the identity \rmref{cycl1} is left to the reader and Eq.~\rmref{cycl3} was already checked in Theorem \ref{backspace}.

Finally, let us see that the operad is not only cyclic but indeed cyclic with multiplication, that is, that $\tau_\diag \circ \mu = \mu$. Indeed,
\begin{equation*}
\begin{split}
(&\tau_\diag \circ \mu)(u,v) 
\\
&= \Big(u_{(1)} S\big(\mu^{(1)}\big(S^{-1}(u_{(3)}v_{(2)}), u_{(2)}\big)\big)\Big) \lact \big(\mu^{(2)}\big(S^{-1}(u_{(3)}v_{(2)}), u_{(2)}\big) v_{(1)}  \otimes 1\big)
\\
&= \Big(u_{(1)} S\big(S^{-1}(u_{(4)}v_{(2)})_{(1)}u_{(2)}\big)\Big) \lact \big((S^{-1}(u_{(4)}v_{(2)})_{(2)}u_{(3)}) v_{(1)}  \otimes 1\big) 
\\
&= \Big(u_{(1)} S\big(S^{-1}(u_{(5)}v_{(3)}) u_{(2)}\big)\Big) \lact \big(S^{-1}(u_{(4)}v_{(2)}) u_{(3)} v_{(1)}  \otimes 1\big) 
\\
&= \big(uv \big) \lact (1  \otimes 1) 
\\
&= \mu(u,v),
\end{split}
\end{equation*}
which is what we wanted.
This concludes the proof.
\end{proof}

Being the structure of a cyclic operad with multiplication given on the diagonal complex that computes Gerstenhaber-Schack cohomology, 
Theorem 1.4 in \cite{Men:BVAACCOHA} then has two important consequences. Firstly, we immediately reproduce Theorem \ref{backspace}:

\begin{cor}
\label{forwardspace}
Together with the cocyclic operator \rmref{duschdas}, the cosimplicial $\K$-module from \rmref{gazette} becomes a cocyclic $\K$-module. 
\end{cor}

Secondly, we obtain:

\begin{cor}
\label{nikolaus}
For any Hopf algebra over a field with involutive antipode, its Gerstenhaber-Schack cohomology groups form a BV algebra.
\end{cor}

\section{The finite dimensional case}
\label{finite}

In case the Hopf algebra $H$ is finitely generated as a module over $\K$, a 
result by Taillefer \cite[Thm.~3.4]{Tai:IHBCOIDHAAGCOTYP} implies that 
\begin{equation}
\label{indeed}
H^\bullet_{GS}(H,H) \simeq \Ext^\bullet_{{\bf YD}}(\K, \K) \simeq \Ext^\bullet_{D(H)}(\K,\K),
\end{equation}
where ${\bf YD}$ denotes the category of (left-left) Yetter-Drinfel'd modules over $H$ and $D(H)$ denotes the Drinfel'd double of $H$, see, for example, \cite[IX.4]{Kas:QG} for a description. 
As $D(H)$ is a Hopf algebra as well (the antipode of which is involutive if this is the case for $H$), 
the cohomology groups on the left hand side in \rmref{indeed} can be computed by the cochain complex $C^\bullet(D(H), \K)$ from \rmref{earlyinthemorning} with differential obtained from the cosimplicial pieces in \rmref{deepatnight2}. In particular, the underlying $\K$-modules in \rmref{earlyinthemorning} constitute an operad with multiplication (see \cite[Eq.~(5.1)]{Kow:WEIABVA}, which is cyclic if the antipode of $H$ is involutive) but the corresponding Gerstenhaber bracket formed as in Eq.~\rmref{naemlichhier} turns out to be zero when descending to cohomology \cite[Rem.~5.4]{Tai:IHBCOIDHAAGCOTYP}.
On the other hand, the customary Hom-tensor adjunction induces in each degree $n \geq 0$ an isomorphism 
\begin{equation}
\label{lonza}
C^n(D(H), \K) \simeq C^n_\diag(H)
\end{equation}
of $\K$-modules and hence the operadic structure on $C^n(D(H), \K)$ induces one on $C^n_\diag(H)$ as well, which, in general, is different from the one given in \rmref{aglio}. 

On top, the isomorphism \rmref{lonza} is in general {\em not} one of cochain complexes and hence the induced operadic structure on  $ C^n_\diag(H)$ does not produce the correct differential $\gd^\diag$ obtained from \rmref{trockenfirma} to compute $H^\bullet_{GS}(H,H)$ when seen as an operad with multiplication in the sense of \rmref{immaginedellacitta`}. Varying the isomorphism \rmref{lonza} by composing it with maps of the type in \rmref{sesam1} or \rmref{mases0} or combinations thereof does not help.

On the other hand, as both cochain complexes $C^\bullet(D(H), \K)$ and $C^\bullet_\diag(H)$ do compute the same cohomology, they are quasi-isomorphic; moreover, as $H$ is finitely generated as a module over $\K$, (\ref{lonza}) implies they must actually be isomorphic as cochain complexes. Cochain complex isomorphisms between  $C^\bullet(D(H), \K)$ and $C^\bullet_\diag(H)$ are, however, neither unique nor canonical. 
Although there is a priori no reason why among various such isomorphisms there should be also a morphism of operads with multiplication, one nevertheless might be tempted to expect 
that the Gerstenhaber bracket from Corollary \ref{fegato} vanishes in cohomology as well in case $H$ is finitely generated as a $\K$-module. 

We prove this by a direct computation:

\begin{theorem}
\label{schnief}
Let $H$ be a finite dimensional Hopf algebra over a field of characteristic zero. Then the Gerstenhaber bracket from Corollary \ref{fegato} vanishes in $H^\bullet_{GS}(H,H)$.
\end{theorem}

\begin{proof}
This can be proven by a direct but annoying computation, which relies on the following observations: first, if $f \in C^p_\diag(H)$ is a cocycle with respect to the differential $\gd_\diag$, one can verify that 
\begin{equation}
\label{minestra}
\sum_{i=1}^{p} (-1)^i f(u^1, \ldots, u^{i-1}, 1, u^{i+1}, \ldots, u^p) = 0, 
\end{equation}
and also that $f(1, \ldots, 1) = 0$ for $p$ odd and  $f(1, \ldots, 1) = 1^{\otimes p}$ (or $\K$-multiples thereof) for $p$ even. From this follows that
\begin{equation}
\label{recitadinatale}
f(u^1_{(1)}, \ldots, u^p_{(1)}) \otimes u^1_{(2)} \cdots u^p_{(2)} = u^1_{(1)} \cdots u^p_{(1)} \otimes f(u^1_{(2)}, \ldots, u^p_{(2)}),
\end{equation}
and inserting this into the operadic composition \rmref{aglio} one can see (with some patience) that $\{f, g \}$ for two cocycles is either zero or a coboundary, hence vanishes in cohomology. For the sake of (paper) compactness, we will elucidate here how to do this for the lowest degree case of two $1$-cocycles, which nevertheless makes clear which kind of algebraic manipulations are used in order to tackle the general case. To give an idea of how the length of the computation increases while the kind of manipulations occurring in it are basically unchanged, the detailed computation of the bracket of a $1$-cocycle and a $2$-cocycle is given in Appendix \ref{favetti}. 

So, let both $f, g \in C^1_\diag(H)$ be both $1$-cocycles, that is, for $u, v$ in $H$ we have
\begin{equation}
\label{funzionepubblica}
f(u_{(1)}) v_{(1)} \otimes u_{(2)} v_{(2)} - \gD( f(uv)) + u_{(1)} v_{(1)} \otimes u_{(2)} f(v_{(2)}) = 0.
\end{equation}
Putting first $u =1$ and then $v = 1$ in Eq.~\rmref{funzionepubblica} and subtracting the two resulting equations one from another leads to
\begin{equation}
\label{hust}
f(u_{(1)})  \otimes u_{(2)} + u_{(1)} \otimes u_{(2)} f(1) 
-
f(1) u_{(1)} \otimes u_{(2)} - u_{(1)} \otimes f(u_{(2)}) = 0.
\end{equation}
Moreover, putting both $u=v=1$ in Eq.~\rmref{funzionepubblica} leads to $f(1) \otimes 1 - \gD(f(1)) + 1 \otimes f(1) = 0$, that is, $f(1)$ is a primitive element. However, a finite dimensional Hopf algebra over a field of characteristic zero does not have any nontrivial primitive elements \cite[Ex.~4.2.16]{DasNasRai:HAAI}; hence, $f(1) = 0$, which is Eq.~\rmref{minestra} in degree one. Putting this back into \rmref{hust}, we obtain \rmref{recitadinatale} in degree one. From this, one further obtains
\begin{equation}
\begin{split}
\label{quipure}
f(u_{(1)})  \otimes u_{(2)} \otimes u_{(3)} 
&= u_{(1)} \otimes \gD\big(f(u_{(2)})\big) 
\\
&=  u_{(1)} \otimes f(u_{(2)}) \otimes u_{(3)}
\\
&=  u_{(1)} \otimes u_{(2)} \otimes f(u_{(3)}),
\end{split}
\end{equation}
where the last two steps again follow from \rmref{funzionepubblica} evaluated for $u=1$ resp.\ $v=1$, along with $f(1) = 0$.
We then have 
\begin{equation*}
\begin{split}
f(u_{(3)}) S^{-1}(u_{(2)}) g(u_{(1)}) &=m^2_{\mathrm{op}}(g\otimes S^{-1}\otimes \id)(u_{(1)}  \otimes u_{(2)} \otimes f(u_{(3)}))\\
&=m^2_{\mathrm{op}}(g\otimes S^{-1}\otimes \id)(u_{(1)} \otimes f(u_{(2)}) \otimes u_{(3)})\\
&=m^2_{\mathrm{op}}(\id\otimes S^{-1}f\otimes \id)(g(u_{(1)}) \otimes u_{(2)} \otimes u_{(3)})\\
&=m^2_{\mathrm{op}}(\id\otimes S^{-1}f\otimes \id)(u_{(1)} \otimes u_{(2)} \otimes g(u_{(3)}))\\
&=m^2_{\mathrm{op}}(\id\otimes S^{-1}\otimes g)(u_{(1)} \otimes f(u_{(2)}) \otimes u_{(3)})\\
&=m^2_{\mathrm{op}}(\id\otimes S^{-1}\otimes g)(f(u_{(1)}) \otimes u_{(2)} \otimes u_{(3)})\\
&=g(u_{(3)}) S^{-1}(u_{(2)}) f(u_{(1)}).
\end{split}
\end{equation*}
Hence, $\{f,g\} = 0$, which concludes the proof.
\end{proof}

From Theorem 5.7 in \cite{FioKow:HBOCANCC} 
then follows at once:

\begin{cor}
\label{schnief2}
For any finite dimensional Hopf algebra over a field of characteristic zero, its Gerstenhaber-Schack cohomology is an $e_3$-algebra by means of the cup product from \rmref{herbst} and the degree $-2$ bracket
\begin{equation}
\label{fidenza}
\{\!\!\{ f,g\}\!\!\}:=(-1)^f (B f)\smallsmile(B g),
\end{equation}
for $f, g \in H^\bullet_{GS}(H,H)$,
where $B$ is Connes' cyclic coboundary obtained from the cocyclic operator \rmref{duschdas} as in Eq.~\rmref{e-mantra}. 
\end{cor}

\begin{rem}
The  $e_3$-algebra structure on $H_{GS}(H,H)$ from Corollary \ref{schnief2} is presumably different from the one exhibited in \cite{GinYal:DTOBHHCAF, Sho:DGCADC, Sho:TOABFATMC} as the latter does not rely on the finite dimensionality of $H$. In particular, the constructions in {\em op.~cit.}~would apply to the enveloping algebra of a nonabelian Lie algebra, whereas Example \ref{pranzo-di-natale} shows that this is not the case for the construction in Corollary \ref{schnief2}.
\end{rem}

%

\appendix

\section{Full proofs of Theorems \ref{japonvuduciel1} and \ref{japonvuduciel2}}
\label{panic}

\begin{proof}[Proof of Theorem \ref{japonvuduciel1}]
It remains to show in full generality that the composition defined in \rmref{aglio} indeed is an operadic one, that is, fulfils the identities in \rmref{danton}.
To prove the vertical composition axiom, that is,
the middle identity in \rmref{danton}, let 
$f \in \Hom_\K(H^{\otimes p}, H^{\otimes p})$,
$g \in  \Hom_\K(H^{\otimes q}, H^{\otimes q})$,
and $h \in  \Hom_\K(H^{\otimes r}, H^{\otimes r})$,
as well as $i \leq j \leq q + i -1$.
We compute
\begin{footnotesize}
\begin{equation*}
\begin{split}
\big(&(f \circ_i g) \circ_j h)(u^1, \ldots, u^{p+q+r-2}) 
\\
&= [(\id^{j-1} \otimes \Delta^{r-1} \otimes \id^{p+q-j-1}) (f \circ_i g)(u^1, \ldots, u^{j-1}, u^j_{(j+2)} \cdots  u^{j+r-1}_{(j+2)}, u^{j+r}_{(j+2)}, \ldots,  u^{p+q+r-2}_{(j+2)})]
\\
& 
\qquad
[S^{-1}(u^{j}_{(j+1)} \cdots u^{p+q+r-2}_{(j+1)}) \lact (u^{j}_{(1)} \cdots u^{p+q+r-2}_{(1)} \otimes \cdots \otimes u^{j}_{(j-1)} \cdots u^{p+q+r-2}_{(j-1)} 
\\
& 
\qquad 
\otimes h(u^{j}_{(j)}, \ldots, u^{j+r-1}_{(j)}) \ract (u^{j+r}_{(j)} \cdots u^{p+q+r-2}_{(j)})) \otimes 1^{\otimes p+q-j-1}]
\\
&=
 \big[\big(\id^{j-1} \otimes \Delta^{r-1} \otimes \id^{p+q-j-1}\big) 
[(\id^{i-1} \otimes \Delta^{q-1} \otimes \id^{p-i}) f(u^1, \ldots, u^{i-1}, 
\\
&
\qquad
u^i_{(i+2)} \cdots   u^{j-1}_{(i+2)} u^j_{(j+i+3)} 
\cdots  u^{i+q+r-2}_{(j+i+3)},  
 u^{i+r+q-1}_{(j+i+4)}, \ldots,  u^{p+q+r-2}_{(j+i+4) })]
\\
& 
\qquad
[S^{-1}(u^i_{(i+1)} \cdots   u^{j-1}_{(i+1)} u^j_{(j+i+2)} \cdots  u^{p+q+r-2}_{(j+i+2)}) \lact 
\\
&
\qquad
(u^i_{(1)} \cdots   u^{j-1}_{(1)} u^j_{(j+2)} 
\cdots u^{p+q+r-2}_{(j+2)}
\otimes \cdots 
\otimes 
u^i_{(i-1)} \cdots   u^{j-1}_{(i-1)} u^j_{(j+i)} 
\cdots u^{p+q+r-2}_{(j+i)}
\\
& 
\qquad 
\otimes g(u^i_{(i)}, \ldots,   u^{j-1}_{(i)}, u^j_{(j+i+1)} \cdots u^{j+r-1}_{(j+i+1)}, 
u^{j+r}_{(j+i+1)},  \ldots, u^{i+q+r-2}_{(j+i+1)}) 
\\
& 
\qquad
\ract (u^{i+r+q-1}_{(j+i+3)} \cdots  u^{p+q+r-2}_{(j+i+3)})) \otimes 1^{\otimes p-i}]
\big]
\big[S^{-1}(u^{j}_{(j+1)} \cdots u^{p+q+r-2}_{(j+1)}) \lact 
\\
&
\qquad
(u^{j}_{(1)} \cdots u^{p+q+r-2}_{(1)} \otimes \cdots \otimes u^{j}_{(j-1)} \cdots u^{p+q+r-2}_{(j-1)}  
\otimes h(u^{j}_{(j)}, \ldots, u^{j+r-1}_{(j)}) 
\\
&
\qquad
\ract (u^{j+r}_{(j)} \cdots u^{p+q+r-2}_{(j)})) \otimes 1^{\otimes p+q-j-1}\big]
\\
&=
\big[\big(\id^{i-1} \otimes \Delta^{q+r-2} \otimes \id^{p-i}\big)  
\\
& 
\qquad
f(u^1, \ldots, u^{i-1}, u^i_{(i+2)} \cdots  
 u^{j-1}_{(i+2)} u^j_{(j+4)} 
\cdots u^{i+q+r-2}_{(j+4)},  
 u^{i+r+q-1}_{(j+4)}, \ldots,  u^{p+q+r-2}_{(j+4)}) \big]
\\
&
\qquad
\big[S^{-1}(u^i_{(i+1)} \cdots   u^{j-1}_{(i+1)} u^j_{(j+3)} \cdots u^{p+q+r-2}_{(j+3)}) \lact (
u^i_{(1)} \cdots   u^{j-1}_{(1)} 
\otimes \cdots 
\otimes 
u^i_{(i-1)} \cdots   u^{j-1}_{(i-1)} 
\\
& 
\qquad 
\otimes (\id^{j-i} \otimes \Delta^{r-1} \otimes \id^{q+i-j-1})(g(u^i_{(i)}, \ldots,   u^{j-1}_{(i)}, u^j_{(j+2)} \cdots u^{j+r-1}_{(j+2)}, 
u^{j+r}_{(j+2)},  \ldots, u^{i+q+r-2}_{(j+2)})) 
\\
&
\qquad
\ract (u^{i+r+q-1}_{(j+2)} \cdots  u^{p+q+r-2}_{(j+2)}))  \otimes 1^{\otimes p-i} \big] 
\big[u^{j}_{(1)} \cdots u^{p+q+r-2}_{(1)} \otimes \cdots \otimes 
u^{j}_{(i-1)} \cdots u^{p+q+r-2}_{(i-1)}
\\
&
\qquad
\otimes 
S^{-1}(u^{j}_{(j+1)} \cdots u^{p+q+r-2}_{(j+1)}) \lact
(u^{j}_{(i)} \cdots u^{p+q+r-2}_{(i)} \otimes \cdots \otimes 
u^{j}_{(j-1)} \cdots u^{p+q+r-2}_{(j-1)} 
\\
& 
\qquad 
\otimes h(u^{j}_{(j)}, \ldots, u^{j+r-1}_{(j)}) \ract (u^{j+r}_{(j)} \cdots u^{p+q+r-2}_{(j)})) \otimes 1^{\otimes p+q-j-1}\big]
\\
&=
\big[\big(\id^{i-1} \otimes \Delta^{q+r-2} \otimes \id^{p-i}\big)  
\\
& 
\qquad
f(u^1, \ldots, u^{i-1}, u^i_{(i+2)} \cdots  
 u^{j-1}_{(i+2)} u^j_{(j+4)} 
\cdots u^{i+q+r-2}_{(j+4)},  
 u^{i+r+q-1}_{(j+5)}, \ldots,  u^{p+q+r-2}_{(j+5)}) \big]
\\
&
\qquad
\big[S^{-1}(u^i_{(i+1)} \cdots   u^{j-1}_{(i+1)} u^j_{(j+3)} \cdots  u^{i+r+q-2}_{(j+3)}  u^{i+r+q-1}_{(j+4)} \cdots u^{p+q+r-2}_{(j+4)}) \lact 
\\
& 
\qquad
\Big(u^{i}_{(1)} \cdots u^{p+q+r-2}_{(1)} \otimes \cdots \otimes 
u^{i}_{(i-1)} \cdots u^{p+q+r-2}_{(i-1)}
\\
&
\qquad
\otimes 
[(\id^{j-i} \otimes \Delta^{r-1} \otimes \id^{q+i-j-1}) g(u^i_{(i)}, \ldots,   u^{j-1}_{(i)}, u^j_{(j+2)} \cdots u^{j+r-1}_{(j+2)}, 
u^{j+r}_{(j+2)},  \ldots, u^{i+q+r-2}_{(j+2)}) 
\\
&
\qquad
\ract (u^{i+r+q-1}_{(j+2)} \cdots  u^{p+q+r-2}_{(j+2)} S^{-1}(u^{j}_{(j+1)} \cdots u^{p+q+r-2}_{(j+1)}))]
\\
&
\qquad
[u^{j}_{(i)} \cdots u^{p+q+r-2}_{(i)} \otimes \cdots \otimes 
u^{j}_{(j-1)} \cdots u^{p+q+r-2}_{(j-1)} 
\otimes h(u^{j}_{(j)}, \ldots, u^{j+r-1}_{(j)}) 
\\
&
\qquad
\ract (u^{j+r}_{(j)} \cdots u^{p+q+r-2}_{(j)}) \otimes (1^{\otimes q-j+i-1} \ract (u^{i+r+q-1}_{(j+3)} \cdots  u^{p+q+r-2}_{(j+3)}))]\Big) \otimes 1^{\otimes p-i}\big]
\\
&=
\big[\big(\id^{i-1} \otimes \Delta^{q+r-2} \otimes \id^{p-i}\big)  
\\
&
\qquad
f(u^1, \ldots, u^{i-1}, u^i_{(i+2)} \cdots  
 u^{j-1}_{(i+2)} u^j_{(j+4)} 
\cdots u^{i+q+r-2}_{(j+4)},  
 u^{i+r+q-1}_{(j+3)}, \ldots,  u^{p+q+r-2}_{(j+3)}) \big]
\\
&
\qquad
\big[S^{-1}(u^i_{(i+1)} \cdots   u^{j-1}_{(i+1)} u^j_{(j+3)} \cdots  u^{i+r+q-2}_{(j+3)}  u^{i+r+q-1}_{(j+2)} \cdots u^{p+q+r-2}_{(j+2)}) \lact 
\\
&
\qquad
\Big(u^{i}_{(1)} \cdots u^{p+q+r-2}_{(1)} \otimes \cdots \otimes 
u^{i}_{(i-1)} \cdots u^{p+q+r-2}_{(i-1)}
\\
&
\qquad
\otimes 
[(\id^{j-i} \otimes \Delta^{r-1} \otimes \id^{q+i-j-1}) g(u^i_{(i)}, \ldots,   u^{j-1}_{(i)}, u^j_{(j+2)} \cdots u^{j+r-1}_{(j+2)}, 
u^{j+r}_{(j+2)},  \ldots, u^{i+q+r-2}_{(j+2)}) 
\\
&
\qquad
\ract S^{-1}(u^{j}_{(j+1)} \cdots u^{i+q+r-2}_{(j+1)})]
[u^{j}_{(i)} \cdots u^{p+q+r-2}_{(i)} \otimes \cdots \otimes 
u^{j}_{(j-1)} \cdots u^{p+q+r-2}_{(j-1)} 
\\
& 
\qquad 
\otimes h(u^{j}_{(j)}, \ldots, u^{j+r-1}_{(j)}) \ract (u^{j+r}_{(j)} \cdots u^{p+q+r-2}_{(j)}) 
\\
&
\qquad
\otimes (1^{\otimes q-j+i-1} \ract (u^{i+r+q-1}_{(j+1)} \cdots  u^{p+q+r-2}_{(j+1)}))]\Big) \otimes 1^{\otimes p-i}\big],
\end{split}
\end{equation*}
\end{footnotesize}
where in case $i = j$ we read this by means of $u_i \cdots u_{j-1} =1$ and $\{u_i, \ldots, u_{j-1}\} = \emptyset$. On the other hand, one has

\begin{footnotesize}
\begin{equation*}
\begin{split}
\big(&f \circ_i (g \circ_{j-i+1} h))(u^1, \ldots, u^{p+q+r-2})
\\
&= [(\id^{i-1} \otimes \Delta^{q+r-2} \otimes \id^{p-i}) f(u^1, \ldots, u^{i-1}, u^i_{(i+2)} \cdots  u^{i+q+r-2}_{(i+2)}, u^{i+q+r-1}_{(i+2)}, \ldots,  u^{p+q+r-2}_{(i+2)})]
\\
& 
\qquad
[S^{-1}(u^{i}_{(i+1)} \cdots u^{p+q+r-2}_{(i+1)}) \lact (u^{i}_{(1)} \cdots u^{p+q+r-2}_{(1)} \otimes \cdots \otimes u^{i}_{(i-1)} \cdots u^{p+q+r-2}_{(i-1)} 
\\
& 
\qquad 
\otimes (g \circ_{j-i+1}h)(u^{i}_{(i)}, \ldots, u^{i+q+r-2}_{(i)}) \ract (u^{i+q+r-1}_{(i)} \cdots u^{p+q+r-2}_{(i)})) \otimes 1^{\otimes p-i}]
\\
&=
\big[(\id^{i-1} \otimes \Delta^{q+r-2} \otimes \id^{p-i}) 
\\
&
\qquad
f(u^1, \ldots, u^{i-1}, u^i_{(i+2)} \cdots   u^{j-1}_{(i+2)} u^j_{(j+4)} 
\cdots u^{i+q+r-2}_{(j+4)}, u^{i+q+r-1}_{(i+2)}, \ldots,  u^{p+q+r-2}_{(i+2)})\big]
\\
& 
\qquad
\big[S^{-1}(u^{i}_{(i+1)} \cdots  u^{j-1}_{(i+1)} u^j_{(j+3)} 
\cdots u^{i+q+r-2}_{(j+3)} u^{i+q+r-1}_{(i+1)} \cdots  u^{p+q+r-2}_{(i+1)}) \lact 
\\
&
\qquad
\Big(u^{i}_{(1)} \cdots u^{p+q+r-2}_{(1)} \otimes \cdots \otimes u^{i}_{(i-1)} \cdots u^{p+q+r-2}_{(i-1)} 
\\
& 
\qquad 
\otimes 
%
%
%
[(\id^{j-i} \otimes \Delta^{r-1} \otimes \id^{q+i-j-1}) g(u^i_{(i)}, \ldots, 
u^{j-1}_{(i)}, u^j_{(j+2)} \cdots u^{j+r-1}_{(j+2)}, u^{j+r}_{(j+2)}, \ldots,  u^{i+q+r-2}_{(j+2)})]
\\
& 
\qquad
[S^{-1}(u^j_{(j+1)} \cdots u^{i+q+r-2}_{(j+1)}) \lact (u^j_{(i)} \cdots u^{i+q+r-2}_{(i)} \otimes \cdots \otimes u^j_{(j-1)} \cdots u^{i+q+r-2}_{(j-1)} 
\\
& 
\qquad 
\otimes h(u^j_{(j)}, \cdots u^{j+r-1}_{(j)}) 
\ract ( u^{j+r}_{(j)} \cdots  u^{i+q+r-2}_{(j)})) \otimes 1^{\otimes q-j+i-1}]
\\
&
\qquad
 \ract 
 (u^{i+q+r-1}_{(i)} \cdots u^{p+q+r-2}_{(i)} )
\Big) \otimes 1^{\otimes p-i}\big]
\\
&=
\big[\big(\id^{i-1} \otimes \Delta^{q+r-2} \otimes \id^{p-i}\big)  
\\
&
\qquad
f(u^1, \ldots, u^{i-1}, u^i_{(i+2)} \cdots  
 u^{j-1}_{(i+2)} u^j_{(j+4)} 
\cdots u^{i+q+r-2}_{(j+4)},  
 u^{i+r+q-1}_{(j+3)}, \ldots,  u^{p+q+r-2}_{(j+3)}) \big]
\\
&
\qquad
\big[S^{-1}(u^i_{(i+1)} \cdots   u^{j-1}_{(i+1)} u^j_{(j+3)} \cdots  u^{i+r+q-2}_{(j+3)}  u^{i+r+q-1}_{(j+2)} \cdots u^{p+q+r-2}_{(j+2)}) \lact 
\\
&
\qquad
\Big(u^{i}_{(1)} \cdots u^{p+q+r-2}_{(1)} \otimes \cdots \otimes 
u^{i}_{(i-1)} \cdots u^{p+q+r-2}_{(i-1)}
\\
&
\qquad
\otimes 
[(\id^{j-i} \otimes \Delta^{r-1} \otimes \id^{q+i-j-1}) g(u^i_{(i)}, \ldots,   u^{j-1}_{(i)}, u^j_{(j+2)} \cdots u^{j+r-1}_{(j+2)}, 
u^{j+r}_{(j+2)},  \ldots, u^{i+q+r-2}_{(j+2)}) 
\\
&
\qquad
\ract S^{-1}(u^{j}_{(j+1)} \cdots u^{i+q+r-2}_{(j+1)})]
[u^{j}_{(i)} \cdots u^{p+q+r-2}_{(i)} \otimes \cdots \otimes 
u^{j}_{(j-1)} \cdots u^{p+q+r-2}_{(j-1)} 
\\
& 
\qquad 
\otimes h(u^{j}_{(j)}, \ldots, u^{j+r-1}_{(j)}) \ract (u^{j+r}_{(j)} \cdots u^{p+q+r-2}_{(j)}) 
\\
&
\qquad
\otimes (1^{\otimes q-j+i-1} \ract (u^{i+r+q-1}_{(j+1)} \cdots  u^{p+q+r-2}_{(j+1)}))]\Big) \otimes 1^{\otimes p-i}\big],
\end{split}
\end{equation*}
\end{footnotesize}
and by staring on this for a (probably not so) little while, one sees that this indeed coincides with the expression obtained above.

The remaining identities in \rmref{danton} that amount to the parallel composition axiom for an operad
are left to the reader. Hence, the family of $\K$-modules 
$C_\diag(H)(n) := \Hom_\K(H^{\otimes n}, H^{\otimes n})$ for $n \geq 0$
defines an operad in $\kmod$.
\end{proof}

\begin{proof}[Proof of Theorem \ref{japonvuduciel2}]
It remains to prove that Eqs.~\rmref{cycl1}--\rmref{cycl3} are fulfilled with respect to the operadic composition \rmref{aglio} and the cocyclic operator \rmref{duschdas}.
To start with,
let us check \rmref{cycl2}.
Note first that for $f \in C_\diag(H)(p)$ and $g \in C_\diag(H)(q)$, one has
\begin{footnotesize}
\begin{equation*}
\begin{split}
&S\big((f \circ_1 g)^{(1)}(u^1, \ldots, u^{p+q-1})\big) \lact \big( 
(f \circ_1 g)^{(2)}(u^1, \ldots, u^{p+q-1}) \otimes \cdots \otimes (f \circ_1 g)^{(p+q-1)}(u^1, \ldots, u^{p+q-1})\big) 
\\
&=
S\big(g^{(1)}(u^1_{(1)}, \ldots, u^{q}_{(1)})  u^{q+1}_{(1)} \cdots  u^{p+q-1}_{(1)} \big) \lact \Big( 
g^{(2)}(u^1_{(1)}, \ldots, u^{q}_{(1)}) u^{q+1}_{(2)} \cdots  u^{p+q-1}_{(2)}  \otimes \cdots
\\
&
\qquad
\otimes g^{(q)}(u^1_{(1)}, \ldots, u^{q}_{(1)}) u^{q+1}_{(q)} \cdots  u^{p+q-1}_{(q)}
\\
&
\qquad
\otimes \big(u^{1}_{(2)} \cdots u^{q}_{(2)}u^{q+1}_{(q+1)} \cdots  u^{p+q-1}_{(q+1)} S\big(f^{(1)}(u^1_{(3)} \cdots u^q_{(3)}, u^{q+1}_{(q+2)}, \ldots, u^{p+q-1}_{(q+2)})\big)\big) \lact 
\\
&
\qquad
\big( 
f^{(2)}(u^1_{(3)} \cdots u^q_{(3)}, u^{q+1}_{(q+2)}, \ldots, u^{p+q-1}_{(q+2)}) \otimes \cdots 
\otimes
f^{(p)}(u^1_{(3)} \cdots u^q_{(3)}, u^{q+1}_{(q+2)}, \ldots, u^{p+q-1}_{(q+2)})\big)\Big).
\end{split}
\end{equation*}
\end{footnotesize}

Using this, one has
\begin{footnotesize}
\begin{equation*}
\begin{split}
(& \tau_\diag (f \circ_1 g))(u^1, \ldots, u^{p+q-1}) 
\\
&= 
 \Big(u^1_{(1)} \cdots u^{p+q-2}_{(1)} S\big((f \circ_1 g)^{(1)}\big(S^{-1}(u^1_{(3)} \cdots  u^{p+q-2}_{(3)} u^{p+q-1}_{(p+q-1)}), u^1_{(2)}, \ldots, u^{p+q-2}_{(2)}\big)\big)\Big) \lact 
\\
&
\qquad
 \big((f \circ_1 g)^{(2)}\big(S^{-1}(u^1_{(3)} \cdots u^{p+q-2}_{(3)} u^{p+q-1}_{(p+q-1)}), u^1_{(2)}, \ldots, u^{p+q-2}_{(2)} \big) u^{p+q-1}_{(1)}  \otimes \cdots 
\\
&
\qquad
\otimes (f \circ_1 g)^{(p+q-1)}\big(S^{-1}(u^1_{(3)} \cdots u^{p+q-2}_{(3)} u^{p+q-1}_{(p+q-1)}), u^1_{(2)}, \ldots, u^{p+q-2}_{(2)}  \big) u^{p+q-1}_{(p+q-2)} \otimes 1\big) 
\\
& 
=
\Big(u^1_{(1)} \cdots u^{p+q-2}_{(1)}S(u^q_{(2)} \cdots u^{p+q-2}_{(2)})  
S\big(g^{(1)}(S^{-1}(u^1_{(5)} \cdots u^{q-1}_{(5)}  u^{q}_{(q+4)}  \cdots u^{p+q-2}_{(q+4)} u^{p+q-1}_{(p+q-1)})_{(1)}, u^1_{(2)}, \ldots, u^{q-1}_{(2)})\big)\Big) \lact 
\\
&
\qquad
\Big( 
g^{(2)}(S^{-1}(u^1_{(5)} \cdots u^{q-1}_{(5)}  u^{q}_{(q+4)}  \cdots u^{p+q-2}_{(q+4)}u^{p+q-1}_{(p+q-1)})_{(1)}, u^1_{(2)}, \ldots, u^{q-1}_{(2)}) u^{q}_{(3)} \cdots  u^{p+q-2}_{(3)} u^{p+q-1}_{(1)}  \otimes \cdots 
\\
&
\qquad
\otimes g^{(q)}(S^{-1}(u^1_{(5)} \cdots u^{q-1}_{(5)}  u^{q}_{(q+4)}  \cdots u^{p+q-2}_{(q+4)}u^{p+q-1}_{(p+q-1)})_{(1)}, u^1_{(2)}, \ldots, u^{q-1}_{(2)}) u^{q}_{(q+1)} \cdots  u^{p+q-2}_{(q+1)} u^{p+q-1}_{(q-1)} 
\\
&
\qquad
\otimes \big(
S^{-1}(u^1_{(5)} \cdots u^{q-1}_{(5)}  u^{q}_{(q+4)}  \cdots u^{p+q-2}_{(q+4)}u^{p+q-1}_{(p+q-1)})_{(2)} u^1_{(3)} \cdots u^{q-1}_{(3)}  u^q_{(q+2)} \cdots u^{p+q-2}_{(q+2)}
\\
&
\qquad
S\big(f^{(1)}(S^{-1}(u^1_{(5)} \cdots u^{q-1}_{(5)}  u^{q}_{(q+4)}  \cdots u^{p+q-2}_{(q+4)}u^{p+q-1}_{(p+q-1)})_{(3)} u^1_{(4)} \cdots u^{q-1}_{(4)}, u^q_{(q+3)}, \ldots, u^{p+q-2}_{(q+3)})\big)\big) \lact 
\\
&
\qquad
\big( 
f^{(2)}(S^{-1}(u^1_{(5)} \cdots u^{q-1}_{(5)}  u^{q}_{(q+4)}  \cdots u^{p+q-2}_{(q+4)}u^{p+q-1}_{(p+q-1)})_{(3)}u^1_{(4)} \cdots u^{q-1}_{(4)},  u^q_{(q+3)}, \ldots, u^{p+q-2}_{(q+3)} )u^{p+q-1}_{(q)}  \otimes \cdots 
\\
&
\qquad
\otimes f^{(p)}(S^{-1}(u^1_{(5)} \cdots u^{q-1}_{(5)}  u^{q}_{(q+4)}  \cdots u^{p+q-2}_{(q+4)}u^{p+q-1}_{(p+q-1)})_{(3)}u^1_{(4)} \cdots u^{q-1}_{(4)},  u^q_{(q+3)}, \ldots, u^{p+q-2}_{(q+3)} )u^{p+q-1}_{(p+q-2)} \otimes 1 \big)
\Big)
\\
&=
 \Big(u^1_{(1)} \cdots u^{q-1}_{(1)}  
S\big(g^{(1)}(S^{-1}(u^1_{(3)} \cdots u^{q-1}_{(3)}  u^{q}_{(q+4)}  \cdots u^{p+q-2}_{(q+4)}u^{p+q-1}_{(p+q+1)}), u^1_{(2)}, \ldots, u^{q-1}_{(2)})\big)\Big) \lact 
\\
&
\qquad
\Big( 
g^{(2)}(S^{-1}(u^1_{(3)} \cdots u^{q-1}_{(3)}  u^{q}_{(q+4)}  \cdots u^{p+q-2}_{(q+4)}u^{p+q-1}_{(p+q+1)}), u^1_{(2)}, \ldots, u^{q-1}_{(2)}) u^{q}_{(1)} \cdots  u^{p+q-1}_{(1)}  \otimes \cdots 
\\
&
\qquad
\otimes g^{(q)}(S^{-1}(u^1_{(3)} \cdots u^{q-1}_{(3)}  u^{q}_{(q+4)}  \cdots u^{p+q-2}_{(q+4)}u^{p+q-1}_{(p+q+1)}), u^1_{(2)}, \ldots, u^{q-1}_{(2)}) u^{q}_{(q-1)} \cdots  u^{p+q-1}_{(q-1)} 
\\
&
\qquad
\otimes \big(
S^{-1}(u^{q}_{(q+3)}  \cdots u^{p+q-2}_{(q+3)}u^{p+q-1}_{(p+q)}) u^q_{(q)} \cdots u^{p+q-2}_{(q)}
\\
&
\qquad
S\big(f^{(1)}(S^{-1}(u^{q}_{(q+2)}  \cdots u^{p+q-2}_{(q+2)}u^{p+q-1}_{(p+q-1)}), u^q_{(q+1)}, \ldots, u^{p+q-2}_{(q+1)})\big)\big) \lact 
\\
&
\qquad
\big( 
f^{(2)}(S^{-1}(u^{q}_{(q+2)}  \cdots u^{p+q-2}_{(q+2)}u^{p+q-1}_{(p+q-1)}),  u^q_{(q+1)}, \ldots, u^{p+q-2}_{(p+q-1)} )u^{p+q-1}_{(q)}  \otimes \cdots 
\\
&
\qquad
\otimes f^{(p)}(S^{-1}(u^{q}_{(q+2)}  \cdots u^{p+q-2}_{(q+2)}u^{p+q-1}_{(p+q-1)}),  u^q_{(q+1)}, \ldots, u^{p+q-2}_{(p+q-1)} )u^{p+q-1}_{(p+q-2)} \otimes 1 \big)
\Big).
\end{split}
\end{equation*}
\end{footnotesize}
On the other hand, one equally tediously computes
\begin{footnotesize}
\begin{equation*}
\begin{split}
(& \tau_\diag g \circ_q \tau_\diag f)(u^1, \ldots, u^{p+q-1}) 
\\
&= [(\id^{q-1} \otimes \Delta^{p-1}) \tau_\diag g(u^1, \ldots, u^{q-1}, u^q_{(q+2)} \cdots  u^{p+q-1}_{(q+2)})]
\\
& 
\qquad
[S^{-1}(u^{q}_{(q+1)} \cdots u^{p+q-1}_{(q+1)}) \lact (u^{q}_{(1)} \cdots u^{p+q-1}_{(1)} \otimes \cdots \otimes u^{q}_{(q-1)} \cdots u^{p+q-1}_{(q-1)} 
\\
& \qquad \otimes \tau_\diag f(u^{q}_{(q)}, \ldots, u^{p+q-1}_{(q)}))]
\\
&= [(\id^{q-1} \otimes \Delta^{p-1}) 
\Big(
 \Big(u^1_{(1)} \cdots u^{q-1}_{(1)} S\big(g^{(1)}\big(S^{-1}(u^1_{(3)} \cdots  u^{q-1}_{(3)}  u^{q}_{(2q+3)} \cdots u^{p+q-2}_{(2q+3)} u^{p+q-1}_{(p+2q)}), u^1_{(2)}, \ldots, u^{q-1}_{(2)}\big)\big)\Big) \lact 
\\
&
\qquad
 \big(g^{(2)}\big(S^{-1}(u^1_{(3)} \cdots  u^{q-1}_{(3)}  u^{q}_{(2q+3)} \cdots u^{p+q-2}_{(2q+3)} u^{p+q-1}_{(p+2q)}), u^1_{(2)}, \ldots, u^{q-1}_{(2)} \big)   u^{q}_{(q+4)} \cdots u^{p+q-2}_{(q+4)} u^{p+q-1}_{(p+q+1)} \otimes \cdots 
\\
&
\qquad 
\otimes (g^{(q)}\big( S^{-1}(u^1_{(3)} \cdots  u^{q-1}_{(3)} u^{q}_{(2q+3)} \cdots u^{p+q-2}_{(2q+3)} u^{p+q-1}_{(p+2q)}), u^1_{(2)}, \ldots, u^{q-1}_{(2)} \big)   u^{q}_{(2q+2)} \cdots u^{p+q-2}_{(2q+2)} u^{p+q-1}_{(p+2q-1)} \otimes 1\big) 
\Big)
]
\\
& 
\qquad
[S^{-1}(u^{q}_{(q+3)} \cdots u^{p+q-2}_{(q+3)} u^{p+q-1}_{(p+q)}) \lact \Big(u^{q}_{(1)} \cdots u^{p+q-1}_{(1)} \otimes \cdots \otimes u^{q}_{(q-1)} \cdots u^{p+q-1}_{(q-1)} 
\\
& \qquad \otimes 
\Big(u^q_{(q)} \cdots u^{p+q-2}_{(q)} S\big(f^{(1)}\big(S^{-1}(u^q_{(q+2)} \cdots  u^{p+q-2}_{(q+2)} u^{p+q-1}_{(p+q-1)}), u^q_{(q+1)}, \ldots, u^{p+q-2}_{(q+1)}\big)\big)\Big) \lact 
\\
&
\qquad
 \big(f^{(2)}\big(S^{-1}(u^q_{(q+2)} \cdots  u^{p+q-2}_{(q+2)} u^{p+q-1}_{(p+q-1)}), u^q_{(q+1)}, \ldots, u^{p+q-2}_{(q+1)}\big) u^{p+q-1}_{(q)}  \otimes \cdots 
\\
&
\qquad 
\otimes f^{(p)}\big(S^{-1}(u^q_{(q+2)} \cdots  u^{p+q-2}_{(q+2)} u^{p+q-1}_{(p+q-1)}), u^q_{(q+1)}, \ldots, u^{p+q-2}_{(q+1)}  \big) u^{p+q-1}_{(p+q-2)} \otimes 1\big)\Big) 
]
\\
&=
 \Big(u^1_{(1)} \cdots u^{q-1}_{(1)}  
S\big(g^{(1)}(S^{-1}(u^1_{(3)} \cdots u^{q-1}_{(3)}  u^{q}_{(q+4)}  \cdots u^{p+q-2}_{(q+4)}u^{p+q-1}_{(p+q+1)}), u^1_{(2)}, \ldots, u^{q-1}_{(2)})\big)\Big) \lact 
\\
&
\qquad
\Big( 
g^{(2)}(S^{-1}(u^1_{(3)} \cdots u^{q-1}_{(3)}  u^{q}_{(q+4)}  \cdots u^{p+q-2}_{(q+4)}u^{p+q-1}_{(p+q+1)}), u^1_{(2)}, \ldots, u^{q-1}_{(2)}) u^{q}_{(1)} \cdots  u^{p+q-1}_{(1)}  \otimes \cdots 
\\
&
\qquad
\otimes g^{(q)}(S^{-1}(u^1_{(3)} \cdots u^{q-1}_{(3)}  u^{q}_{(q+4)}  \cdots u^{p+q-2}_{(q+4)}u^{p+q-1}_{(p+q+1)}), u^1_{(2)}, \ldots, u^{q-1}_{(2)}) u^{q}_{(q-1)} \cdots  u^{p+q-1}_{(q-1)} 
\\
&
\qquad
\otimes \big(
S^{-1}(u^{q}_{(q+3)}  \cdots u^{p+q-2}_{(q+3)}u^{p+q-1}_{(p+q)}) u^q_{(q)} \cdots u^{p+q-2}_{(q)}
\\
&
\qquad
S\big(f^{(1)}(S^{-1}(u^{q}_{(q+2)}  \cdots u^{p+q-2}_{(q+2)}u^{p+q-1}_{(p+q-1)}), u^q_{(q+1)}, \ldots, u^{p+q-2}_{(q+1)})\big)\big) \lact 
\\
&
\qquad
\big( 
f^{(2)}(S^{-1}(u^{q}_{(q+2)}  \cdots u^{p+q-2}_{(q+2)}u^{p+q-1}_{(p+q-1)}),  u^q_{(q+1)}, \ldots, u^{p+q-2}_{(p+q-1)} )u^{p+q-1}_{(q)}  \otimes \cdots 
\\
&
\qquad
\otimes f^{(p)}(S^{-1}(u^{q}_{(q+2)}  \cdots u^{p+q-2}_{(q+2)}u^{p+q-1}_{(p+q-1)}),  u^q_{(q+1)}, \ldots, u^{p+q-2}_{(p+q-1)} )u^{p+q-1}_{(p+q-2)} \otimes 1 \big)
\Big),
\end{split}
\end{equation*}
\end{footnotesize}
which by half an hour staring on it hopefully without getting mad one can realise that this is the same expression as above.

Checking the identity \rmref{cycl1} is left to the reader and Eq.~\rmref{cycl3} was already checked in Theorem \ref{backspace}. This concludes the proof.
\end{proof}

\section{Addendum to the proof of Theorem \ref{schnief}}
\label{favetti}

In this appendix, we illustrate how the computational complexity increases if one passes to the next simple situation: we show that $\{f,g\} = 0$ in Gerstenhaber-Schack cohomology for a finite dimensional Hopf algebra over a field of characteristic zero, where $f \in C^1_\diag(H)$ is a $1$-cocycle and 
$g \in C^2_\diag(H)$ a $2$-cocycle.
The cocycle condition in degree two explicitly reads:
\begin{equation}
\begin{split}
\label{haydn}
& g(u_{(1)}, v_{(1)}) \ract w_{(1)}  \otimes u_{(2)} v_{(2)}w_{(2)} - (\gD \otimes \id) g(uv, w) 
\\
&
\qquad
+ (\id \otimes \gD)g(u,vw)
- u_{(1)} v_{(1)} w_{(1)}\otimes u_{(2)} \lact g(v_{(2)}, w_{(2)}) = 0,
\end{split}
\end{equation}
for $u,v,w \in H$.
Putting $u = v = w = 1$ in \rmref{haydn} and applying either $(\id \otimes \gve \otimes \gve)$ or  $(\gve \otimes \gve \otimes \id)$ to it yields the equations
$
g^{(1)}(1,1) \gve(g^{(2)}(1,1)) = \gve(g^{(1)}(1,1)) \gve(g^{(2)}(1,1)) = \gve(g^{(1)}(1,1)) g^{(2)}(1,1),
$
hence $g(1,1) = 1 \otimes 1$ (or a $\K$-multiple thereof). Putting $v = w = 1$ in \rmref{haydn} and applying $(\id \otimes \gve \otimes \gve)$ to it gives
$
g^{(1)}(u,1) \gve(g^{(2)}(u,1)) = u \big((\gve \otimes \gve) g(1,1)\big), 
$
whereas with $u = v = 1$, one obtains 
$
\gve(g^{(1)}(1,u))  g^{(2)}(1, u) 
= u ((\gve \otimes \gve) g(1,1)).
$
Setting again $v = w = 1$ resp.\ $u = v = 1$ in \rmref{haydn}, applying $(\id \otimes \id \otimes \gve)$ to it, and using the identities just obtained, yields 
$$
g(u,1) = ((\gve \otimes \gve) g(1,1)) \gD u = g(1,u), 
$$
which is Eq.~\rmref{minestra} in degree two. With similar manipulations one can even show that 
$u_{(1)} \otimes g(u_{(2)},1) = [\gD u][g(1,1)]  = [g(1,1)][\gD u]
= g(1,u_{(1)}) \otimes u_{(2)}$ and from this 
\begin{equation}
\label{trifonov}
(\Delta \otimes \id)g(u,v) = (\id \otimes \Delta)g(u,v).
\end{equation}
Take Eq.~\rmref{haydn} again and subtract one from another the three equations 
resulting from setting $u$ resp.\ $v$ resp.\ $w$ to be equal to $1$.
By inserting all of the above into the outcome of this subtraction, one {\em finally} ends up with
\begin{equation}
\label{scatola1}
g(u_{(1)}, v_{(1)}) \otimes u_{(2)} v_{(2)} = u_{(1)}v_{(1)} \otimes g(u_{(2)}, v_{(2)}), 
\end{equation}
which is Eq.~\rmref{recitadinatale} in degree two. Written in components $g(u,v) = g^{(1)}(u,v) \otimes g^{(2)}(u,v) \in H \otimes H$, this means:
\begin{equation}
\label{scatola2}
\begin{split}
& g^{(1)}(u_{(1)}, v_{(1)}) \otimes g^{(2)}(u_{(1)}, v_{(1)}) \otimes  u_{(2)} v_{(2)} 
\\
&
\quad
= u_{(1)}v_{(1)} \otimes g^{(1)}(u_{(2)}, v_{(2)}) \otimes  g^{(2)}(u_{(2)}, v_{(2)}),
\end{split}
\end{equation}
which is the crucial identity we are going to repeatedly use now:
the Gerstenhaber bracket of $f$ and $g$ reads as
$\{f,g\} = f \circ_1 g - g \circ_1 f - g \circ_2 f$,
where
\begin{equation}
\label{pelikan}
\begin{array}{rcl}
(f \circ_1 g)(u, v) &=& [\Delta f(u_{(3)}v_{(3)})][ S^{-1}(u_{(2)}v_{(2)}) \lact g(u_{(1)}, v_{(1)})], 
\\
(g \circ_1 f)(u, v) &=&  [g(u_{(3)}, v_{(3)})][ S^{-1}(u_{(2)}v_{(2)}) \lact f(u_{(1)}) \ract v_{(1)} \otimes  1],
\\
(g \circ_2 f)(u, v) &=&
 [g(u , v_{(4)})] [S^{-1}(v_{(3)}) \lact (v_{(1)} \otimes  f(v_{(2)}))].
\end{array}
\end{equation}
Using \rmref{recitadinatale} for $f$, the last one of these equations becomes:
$$
(g \circ_2 f)(u, v) 
=
 [g(u , v_{(4)})] [S^{-1}(v_{(3)}) \lact (f(v_{(1)}) \otimes  v_{(2)})]
=
 [g(u , v_{(4)})] [S^{-1}(v_{(2)}) f(v_{(1)}) \otimes  1],
$$
whereas using the cocycle condition \rmref{funzionepubblica} for $f$ along with \rmref{quipure} and \rmref{scatola2}, the first equation in \rmref{pelikan} becomes:
\begin{small}
\begin{equation*}
\begin{split}
&(f \circ_1 g)(u, v) 
\\
&= [f(u_{(3)})v_{(3)} \otimes u_{(4)} v_{(4)} + u_{(3)} v_{(3)} \otimes u_{(4)} f(v_{(4)})][ S^{-1}(u_{(2)}v_{(2)}) \lact g(u_{(1)}, v_{(1)})] 
\\
&= 
[f(u_{(3)})v_{(3)} \otimes u_{(4)} v_{(4)} + u_{(3)} v_{(3)} \otimes u_{(4)} f(v_{(4)})]
\\
&
\qquad
[ S^{-1}(g^{(2)}(u_{(2)}, v_{(2)})) \lact (u_{(1)} v_{(1)} \otimes g^{(1)}(u_{(2)}, v_{(2)}))]
\\
 &= 
[f(u_{(3)})v_{(3)} \otimes u_{(4)} v_{(4)} + u_{(3)} v_{(3)} \otimes u_{(4)} f(v_{(4)})]
\\
&
\qquad
[ S^{-1}(g^{(2)}(u_{(2)}, v_{(2)})) \lact (u_{(1)} v_{(1)} \otimes g^{(1)}(u_{(2)}, v_{(2)}))]
\\
&= 
[\Delta(u_{(3)} v_{(3)})] [ S^{-1}(g^{(2)}(u_{(2)}, v_{(2)})) \lact (f(u_{(1)}) v_{(1)} \otimes g^{(1)}(u_{(2)}, v_{(2)}))]
\\
&
\quad
+
[\Delta(u_{(3)} v_{(3)})] [ S^{-1}(g^{(2)}(u_{(2)}, v_{(2)})) \lact (u_{(1)} f(v_{(1)}) \otimes g^{(1)}(u_{(2)}, v_{(2)}))].
\end{split}
\end{equation*}
\end{small}
From this it is now clear that for having $\{f,g\} = 0$ it is enough to show:
\begin{footnotesize}
$$
[\Delta(u_{(2)}v_{(2)})][ S^{-1}(g^{(2)}(u_{(1)}, v_{(1)})) \lact (1 \otimes g^{(1)}(u_{(1)}, v_{(1)}))]
= [g(u_{(2)}, v_{(2)})][ S^{-1}(u_{(1)} v_{(1)}) \otimes 1],
$$
\end{footnotesize}
which is seen using repeatedly \rmref{trifonov}, \rmref{scatola1}, and the two-sided counitality of the underlying coalgebra, 
that is, $\gve(u_{(1)}) u_{(2)} = u_{(1)} \gve(u_{(2)})$: 
\begin{small}
\begin{equation*}
\begin{split}
&[\Delta(u_{(2)}v_{(2)})][ S^{-1}(g^{(2)}(u_{(1)}, v_{(1)})) \lact (1 \otimes g^{(1)}(u_{(1)}, v_{(1)}))]
\\
&= 
[\Delta(u_{(2)}v_{(2)})][ S^{-1}(g^{(2)}(u_{(1)}, v_{(1)})_{(2)}) \otimes  S^{-1}(g^{(2)}(u_{(1)}, v_{(1)})_{(1)}) g^{(1)}(u_{(1)}, v_{(1)})]
\\
&= 
[\Delta(u_{(2)}v_{(2)})][ S^{-1}(g^{(2)}(u_{(1)}, v_{(1)})) \otimes  \gve(g^{(1)}(u_{(1)}, v_{(1)}))]
\\
&=
[\Delta(g^{(2)}(u_{(2)},v_{(2)}))][ S^{-1}(g^{(1)}(u_{(2)}, v_{(2)})) \otimes  \gve(u_{(1)} v_{(1)})]
\\
&=
[\Delta(g^{(2)}(u,v))][ S^{-1}(g^{(1)}(u, v)) \otimes  1]
\\
&=
[\Delta(g^{(2)}(u_{(1)},v_{(1)}))\gve(u_{(2)}v_{(2)}][ S^{-1}(g^{(1)}(u_{(1)}, v_{(1)})) \otimes  1]
\\
&=
[\Delta(g^{(1)}(u_{(2)},v_{(2)}))\gve(g^{(2)}(u_{(2)}, v_{(2)})][ S^{-1}(u_{(1)} v_{(1)}) \otimes  1]
\\
&=
[g^{(1)}(u_{(2)},v_{(2)}) \otimes g^{(2)}(u_{(2)}, v_{(2)})_{(1)} \gve(g^{(2)}(u_{(2)}, v_{(2)})_{(2)})][ S^{-1}(u_{(1)} v_{(1)}) \otimes  1]
\\
&= [g(u_{(2)}, v_{(2)})][ S^{-1}(u_{(1)} v_{(1)}) \otimes 1].
\end{split}
\end{equation*}
\end{small}
Hence, we have $\{f, g\} = 0$ for a $1$-cocycle $f$ and a $2$-cocycle $g$, as claimed.

\section{Cocyclic and para-cocyclic modules}
\label{defilambda}
{\em Para-cocyclic $\K$-modules} \cite[p.~164]{GetJon:TCHOCPA} generalise 
cocyclic $\K$-modules by dropping the
condition that the cyclic operator
implements an action of
$\mathbb{Z}/(n+1)\mathbb{Z}$ on the
degree $n$ part. Thus, a para-cocyclic $\K$-module is a
cosimplicial $\K$-module
$(C^\bullet,\delta_\bull,\sigma_\bull)$
together with $\K$-linear maps 
$\tau_n : C^n \rightarrow C^n$
subject to
\begin{equation*}
\label{belleville}
\begin{array}{cc}
\begin{array}{rcll}
\tau_n \circ \gd_i &\!\!\!\!\!=&\!\!\!\!\! \left\{\!\!\!
\begin{array}{l}
\gd_{i-1}\circ \tau_{n-1} \\
 \gd_n 
\end{array}\right. & \!\!\!\!\!\!\!\!\! 
\begin{array}{l} \mbox{if} \ 1
\leq i \leq n, 
 \\ \mbox{if} \ i = 0,
 \end{array}
\end{array}
&
\begin{array}{rcll}
\tau_n \circ \sigma_i &\!\!\!\!\!=&\!\!\!\!\! \left\{\!\!\!
\begin{array}{l}
\sigma_{i-1} \circ \tau_{n+1} \\
 \sigma_n \circ \tau^2_{n+1} 
\end{array}\right. & \!\!\!\!\!\!\!\!\!
\begin{array}{l} \mbox{if} \ 1 \leq i
 \leq n,  
\\ \mbox{if} \ i = 0. 
\end{array} 
\end{array}
\end{array}
\end{equation*}
These relations imply that $\tau_n^{n+1}$
commutes with all the cofaces and
codegeneracies. 
If one additionally has 
$$
\tau_n^{n+1} = \id,
$$
the para-cocyclic module is called {\em cyclic}. 
As common, we also introduce the {\em cyclic coboundary} 
\begin{equation}
\label{e-mantra}
B := N \gs_{-1} (1 - (-1)^n \tau_{n+1}),
\end{equation}
where $\gs_{-1} := \gs_n \tau_{n+1}$ is the {\em extra degeneracy} and $N := \sum^n_{i=0} (-1)^{i+n} \tau^i_n$ is the {\em norm operator}.


\section{Algebraic operads}
\label{pamukkale}

\subsection{Operads and Gerstenhaber algebras}
\label{pamukkale1}
By a {\em non-$\gS$ operad} $\cO$ in the category 
of $\K$-modules we mean a sequence $\{\cO(n)\}_{n \geq 0}$ of $\K$-modules 
equipped with $\K$-bilinear operations $\circ_i: \cO(p) \otimes \cO(q) \to \cO({p+q-1})$, $i = 1, \ldots, p$,
that respect the following compatibility (or associativity) relations (see, for example, \cite[Def.~1.1]{Mar:MFO}): 
\begin{eqnarray}
\label{danton}
\nonumber
\gvf \circ_i \psi &=& 0 \qquad \qquad \qquad \qquad \qquad \! \mbox{if} \ p < i \quad \mbox{or} \quad p = 0, \\
(\varphi \circ_i \psi) \circ_j \chi &=& 
\begin{cases}
(\varphi \circ_j \chi) \circ_{i+r-1} \psi \qquad \mbox{if} \  \, j < i, \\
\varphi \circ_i (\psi \circ_{j-i +1} \chi) \qquad \hspace*{1pt} \mbox{if} \ \, i \leq j < q + i, \\
(\varphi \circ_{j-q+1} \chi) \circ_{i} \psi \qquad \mbox{if} \ \, j \geq q + i.
\end{cases}
\end{eqnarray}

An operad is called {\em unital} if there is an {\em identity} $\mathbb{1} \in \cO(1)$ such that 
$
\gvf \circ_i \mathbb{1} = \mathbb{1} \circ_1 \gvf = \gvf
$ 
for all $\gvf \in \cO(p)$ and $i \leq p$, and the operad is {\em with multiplication} if there exist elements called the {\em multiplication}  $\mu \in \cO(2)$ and the {\em unit} $e \in \cO(0)$ such that $\mu \circ_1 \mu = \mu \circ_2 \mu$ and 
$\mu \circ_1 e = \mu \circ_2 e = \mathbb{1}$. We denote such an object by the triple $(\cO, \mu, e)$.

An operad with multiplication  $(\cO, \mu, e)$ naturally defines a cosimplicial $\K$-module \cite{McCSmi:ASODHCC} given by $\cO^p := \cO(p)$ with faces and degeneracies $\gvf \in \cO(p)$ given by $\gd_0 \gvf := \mu \circ_2 \gvf$, $\gd_i \gvf := \gvf \circ_i \mu$ for $i = 1, \ldots, p$, and $\gd_{p+1} \gvf := \mu \circ_1 \gvf$, along with $\sigma_j(\gvf) := \gvf \circ_{j+1} e$ for $j = 0, \ldots, p-1$. Hence (by the Dold-Kan correspondence), one obtains a cochain complex,
which we will denote by the same symbol $\cO$, with $\cO(n)$ in degree $n$, 
 with differential $\gd: \cO(n) \to \cO({n+1})$ given by $\gd := \sum^{n+1}_{i=0} (-1)^i \gd_i$, and cohomology defined by 
$
H^\bullet(\cO) := H(\cO, \gd).
$ 

The {\em cup product}, on the other hand, is defined as 
\begin{equation}
\label{nuvole}
\psi \smallsmile \gvf := (\mu \circ_2 \psi) \circ_1 \gvf \in \cO(p+q),
\end{equation}  
for $\gvf \in \cO(p)$ and $\psi \in \cO(q)$, and then $(\cO, \smallsmile, \gd)$ yields a dg algebra. One also defines the {\em Gerstenhaber bracket} by means of 
\begin{equation}
\label{naemlichhier}
{\{} \varphi,\psi \}
:= \varphi\{\psi\} - (-1)^{(p-1)(q-1)} \psi\{\varphi\}, 
\end{equation}
where $\varphi\{\psi\} := \sum^{p}_{i=1}
        (-1)^{(q-1)(i-1)} \varphi \circ_i \psi  \in \cO({p+q-1})$ is the {\em brace} \cite{Ger:TCSOAAR, Get:BVAATDTFT}, the sum over all possible partial compositions.
 Notice that one has $\{\mu,\mu\} = 0$ as well as 
\begin{equation}
\label{immaginedellacitta`}
\gd \varphi = (-1)^{p+1} \{\mu, \varphi \}.
\end{equation}
Descending to cohomology, it is not too difficult to verify that $(H^\bullet(\cO), \smallsmile, \{\cdot, \cdot\})$ forms a Gerstenhaber algebra \cite{Ger:TCSOAAR, GerSch:ABQGAAD, McCSmi:ASODHCC}.

\subsection{Cyclic operads and Batalin-Vilkoviski\u\i\ algebras}
\label{pamukkale2}
A {\em cyclic operad} $(\cO, \tau)$, probably first introduced in \cite{GetKap:COACH}, is an operad $\cO$ in the sense of the preceding paragraph that comes with a degree-preserving linear map $\tau: \cO(p) \to \cO(p)$ for $p \in \N$ such that 
\begin{eqnarray}
\label{cycl1}
\tau(\gvf \circ_i \psi) &=& \tau\gvf \circ_{i-1} \psi, \qquad \mbox{for} \ p \geq 2, q \geq 0, 2 \leq i \leq p, \\
\label{cycl2}
\tau(\gvf \circ_1 \psi) &=& \tau \psi \circ_q \tau\gvf, \qquad \mbox{for} \ p \geq 1, q \geq 1, \\ 
\label{cycl3}
\tau^{p+1} &=& \id_{\cO(p)},
\end{eqnarray}
for each $\gvf \in O(p)$ and $\psi \in \cO(q)$. A {\em cyclic operad with multiplication}  $(\cO, \mu, e, \tau)$ as introduced in \cite{Men:BVAACCOHA}
is both an operad with multiplication and a cyclic operad, such that 
$$
\tau \mu = \mu.
$$
Again, descending to cohomology, it is a not too easy check that the $(H^\bullet(\cO), \smallsmile, \{\cdot, \cdot\}, B)$ forms a Batalin-Vilkoviski\u\i\ algebra 
\cite{Men:BVAACCOHA}, that is, a Gerstenhaber algebra whose bracket is ``generated'' by the cyclic coboundary $B$ in the sense of
$$
\{\gvf, \psi\} = - (-1)^\gvf B \gvf \smallsmile \psi - \gvf \smallsmile B \psi + (-1)^\phi B(\gvf \smallsmile \psi).
$$

\providecommand{\bysame}{\leavevmode\hbox to3em{\hrulefill}\thinspace}
\providecommand{\MR}{\relax\ifhmode\unskip\space\fi M`R }
\providecommand{\MRhref}[2]{%
  \href{http://www.ams.org/mathscinet-getitem?mr=#1}{#2}}
\providecommand{\href}[2]{#2}

\end{document}